\documentclass[11pt,reqno]{amsart}
\usepackage{amsmath,amssymb,mathrsfs}
\usepackage{graphicx,cite}
\usepackage{subfigure}
\usepackage{epstopdf}
\usepackage{graphics} 
\usepackage{epsfig} 
\usepackage{tikz}
\usepackage{paralist}
\usepackage{booktabs}
\usepackage{enumerate}
\usepackage{enumitem}
\usepackage{mdwlist}

\newtheorem{theorem}{Theorem}[section]

\newtheorem{remark}[theorem]{Remark}

\numberwithin{equation}{section}

\newcommand{\me}{\mathrm{e}}
\newcommand{\mi}{\mathrm{i}}

\newcommand{\kp}{\kappa_\mathfrak{p}}
\newcommand{\ks}{\kappa_\mathfrak{s}}
\newcommand{\bg}{\boldsymbol{g}}

\newcommand{\bq}{\boldsymbol{q}}
\newcommand{\bw}{\boldsymbol{w}}

\begin{document}

\title[A spectral boundary integral method for elastic scattering]{A spectral boundary integral method for the elastic obstacle scattering problem in three dimensions}

\author{Heping Dong}
\address{School of Mathematics, Jilin University, Changchun,  Jilin 130012, P. R. China}
\email{dhp@jlu.edu.cn}

\author{Jun Lai}
\address{School of Mathematical Sciences, Zhejiang University
	Hangzhou, Zhejiang 310027, China}
\email{laijun6@zju.edu.cn}

\author{Peijun Li}
\address{Department of Mathematics, Purdue University, West Lafayette, Indiana
	47907, USA}
\email{lipeijun@math.purdue.edu}



\begin{abstract}
In this paper, we consider the scattering of a plane wave by a rigid obstacle embedded in a homogeneous and isotropic elastic medium in three dimensions. Based on the Helmholtz decomposition, the elastic scattering problem is reduced to a coupled boundary value problem for the Helmholtz and Maxwell equations. A novel system of boundary integral equations is formulated and a spectral boundary integral method is developed for the coupled boundary value problem. Numerical experiments are presented to demonstrate the superior performance of the proposed method. 
\end{abstract}

\maketitle

\section{Introduction}

The scattering problems for elastic waves have attracted considerable attention due to the significant applications in diverse scientific areas such as nondestructive testing, medical imaging, and seismic exploration \cite{LL-86, ABG-15}. Although many mathematical and computational results are available, it still presents a challenging question on accurate computing of the scattering problems for elastic waves, especially in three dimensions, due to the complexity of the underlying equation. This paper is concerned with a numerical solution for the time-harmonic elastic scattering problem of a rigid obstacle embedded in a homogeneous and isotropic elastic medium in three dimensions. The goal is to develop a spectral boundary integral method for the elastic obstacle scattering problem. 

Compared with the finite element or finite difference methods, the method of boundary integral equations has two intrinsic advantages: it is only required to discretize the boundary of the domain and the radiation condition at infinity is satisfied automatically\cite{PV-JASA,MP-book1986}. However, it also brings an extra difficulty that boundary integrals are usually singular and their accurate numerical approximation is highly involved, especially for three dimensional geometries. Over the years, various methods of boundary integral equations have been proposed to solve the three-dimensional elastic scattering problems. A high order singular integral quadrature method with GMRES was developed in \cite{BronoYin2020} for the elastic scattering problems with the Dirichlet and Neumann boundary conditions on closed and open surfaces. In \cite{BLR2014}, the elastic wave scattering of a time-harmonic incident wave that impinges on a penetrable obstacle was considered, and the singular integral was discretized by the use of partition of unity. Based on the fact that for analytic functions on a smooth closed surface that is isomorphic to a sphere and the interpolation based on spherical harmonics gives spectral accuracy, a high order method for singular integrals in the boundary integral equation was developed in \cite{GG2004} and \cite{GH2008} for the acoustic wave equation and Maxwell's equations, respectively. In \cite{Louer2014}, a high order spectral method was proposed for solving  elastic obstacle scattering problem with the Dirichlet or Neumann boundary condition by directly utilizing the Green function of the three-dimensional elastic wave equation. 

It is worth mentioning that the Green function of the elastic wave equation is a second order tensor and the singularity is tedious to be separated in the computation of boundary integral equations, especially for the Neumann boundary condition and 
the three-dimensional problem \cite{BXY2017, BLR2014, BronoYin2020, Louer2014, TC2007, LR1993}. To bypass this complexity, we
employ the Helmholtz decomposition by introducing one scalar potential function and one vector potential function to split the displacement of the elastic wave field into the compressional and shear wave components. The two wave components, one of which satisfies the three-dimensional Helmholtz equation and one of which satisfies the Maxwell equation, are coupled at the boundary of the obstacle. Therefore, the boundary value problem of the elastic wave equation is converted equivalently into a coupled boundary value problem of the Helmholtz and Maxwell equations for the potentials. Such a decomposition greatly reduces the complexity for the computation of the elastic scattering problem. Similar techniques have also been successfully applied to many other problems such as the unsteady and incompressible flow, the two-dimensional elastic scattering, and inverse scattering problems \cite{GJ2019,DLL2019,DLL2020,DLL2021,LY2019,YLLY19}. 

In this work, by making use of the Helmholtz decomposition, the elastic obstacle scattering problem is reduced to a coupled boundary value problem, which is shown to have a unique solution. Based on the potential theory for the Helmholtz and Maxwell equations, a system of boundary integral equations is formulated for the coupled boundary value problem, and the uniqueness of the solution is discussed for the boundary integral formulation. For the numerical discretization, we adopt the Galerkin method and use the surface differential operators and Stokes' formula to reduce the strong singular operators to weakly singular ones. The approach leads to a high order full-discrete scheme which is similar to the one developed for the acoustic obstacle scattering problem in three dimensions \cite{GG2004}. It should be emphasized that all operations in the full discretization scheme are scalar, which greatly simplify the numerical implementation. Numerical experiments are provided for various geometries and different wavenumbers to demonstrate the superior performance of the proposed method. 

To summarize, the paper contains three contributions: 

\begin{enumerate}
	\item propose a novel boundary integral formulation for the coupled boundary value problem via the Helmholtz decomposition;
	
	\item regularize the singularity of the boundary integral by making use of the Stokes' formula and surface differential operators;
		
	\item develop a spectral method for the approximation of the coupled boundary integral equations by using spherical harmonics.
\end{enumerate}

The paper is organized as follows. In Section 2, we introduce the elastic scattering problem and reduce it to a coupled boundary value problem by using the Helmholtz decomposition. In Section 3, the system of coupled boundary integral equations is presented and the uniqueness is examined for the solution. Section 4 gives the spherical parameterization of the surface integral and discusses the regularization of the strong singular operators. The full-discrete spectral scheme is proposed in Section 5 for the system of the coupled boundary integral equations. Numerical experiments are shown in Section 6 to demonstrate the effectiveness of the proposed method. The paper concludes with some general remarks in Section 7. 


\section{Problem formulation}

Consider a three-dimensional elastically rigid obstacle, which is given as a bounded domain $D\subset\mathbb R^3$ with analytic boundary $\Gamma_D$. Denote by $\nu$ the unit normal vector and $\tau_1, \tau_2$ the orthonormal tangential vectors on $\Gamma_D$, respectively. The exterior domain $\mathbb{R}^3\setminus \overline{D}$ is assumed to be filled with a homogeneous and isotropic elastic medium with a unit mass density. 

Let the obstacle be illuminated by a time-harmonic wave given explicitly by either the compressional plane wave
$\boldsymbol{u}^i(x)=\boldsymbol{d} \mathrm{e}^{\mathrm{i} \kappa_{\mathfrak p}\boldsymbol{d}\cdot x}$ or the shear plane wave
$\boldsymbol{u}^i(x)=\boldsymbol{d}\times\boldsymbol{p} \mathrm{e}^{\mathrm{i}\kappa_{\mathfrak s} \boldsymbol{d}\cdot x}$, 
where $\boldsymbol{d}=(\sin\theta\cos\varphi, \sin\theta\sin\varphi, \cos\theta)^\top$ is the unit propagation direction vector with $\theta\in[0, \pi], \varphi\in[0, 2\pi)$ being the incident angles, $\boldsymbol{p}$ is the unit polarization vector satisfying $\boldsymbol{p}\cdot\boldsymbol{d}=0$,  
and
\[
\kappa_{\mathfrak p}=\frac{\omega}{\sqrt{\lambda+2\mu}},\quad \kappa_{\mathfrak s}=\frac{\omega}{\sqrt{\mu}}
\]
are the compressional and shear wavenumbers, respectively. Here $\omega>0$ is the angular frequency and $\lambda, \mu$ are the
Lam\'{e} constants satisfying $\mu>0, \lambda+\mu>0$. It can be verified that the incident wave $\boldsymbol{u}^i$ satisfies the Navier equation
\begin{equation*}
 \mu\Delta\boldsymbol{u}^i+(\lambda+\mu)\nabla\nabla\cdot\boldsymbol{u}^i
+\omega^2\boldsymbol{u}^i=0\quad {\rm in} ~ \mathbb{R}^3. 
\end{equation*}

The displacement of the total field $\boldsymbol{u}$ satisfies the Navier equation
\begin{equation*}
\mu\Delta\boldsymbol{u}+(\lambda+\mu)\nabla\nabla\cdot\boldsymbol{u}
+\omega^2\boldsymbol{u}=0\quad {\rm in} ~ \mathbb{R}^3\setminus \overline{D}. 
\end{equation*}
The total field $\boldsymbol u$ consists of the incident field $\boldsymbol{u}^i$ and the scattered field $\boldsymbol v$, i.e., 
\begin{equation*}
\boldsymbol u=\boldsymbol{u}^i +\boldsymbol v. 
\end{equation*}
Since the obstacle is assumed to be rigid, we have  
\[
\boldsymbol{u}=0\quad {\rm on}~\Gamma_D.
\]
Hence the scattered field $\boldsymbol v$ satisfies
the boundary value problem
\begin{equation}\label{scatteredfield}
\begin{cases}
\mu\Delta\boldsymbol{v}+(\lambda+\mu)\nabla\nabla\cdot\boldsymbol{v}
+\omega^2\boldsymbol{v}=0\quad &{\rm in}~
\mathbb{R}^3\setminus\overline{D},\\
\boldsymbol{v}=-\boldsymbol{u}^i\quad &{\rm on}~\Gamma_D.
\end{cases}
\end{equation}

For any solution $\boldsymbol v$ of the Navier equation in 
\eqref{scatteredfield}, it has the Helmholtz decomposition 
\begin{equation}\label{HelmDeco}
\boldsymbol{v}=\boldsymbol v_{\mathfrak p}+\boldsymbol v_{\mathfrak s},
\end{equation}
where 
\[
\boldsymbol v_{\mathfrak p}=\nabla\phi,\quad \boldsymbol v_{\mathfrak s}={\bf curl}\boldsymbol\psi, \quad \nabla\cdot\boldsymbol\psi=0.
\]
Here $\boldsymbol v_{\mathfrak p}$ and $\boldsymbol v_{\mathfrak s}$ are known as the compressional and shear wave components of $\boldsymbol v$, respectively. Combining \eqref{scatteredfield} and \eqref{HelmDeco}, we may obtain the
Helmholtz equation for the scalar potential $\phi$ and the Maxwell equation for the vector potential $\boldsymbol\psi$, respectively: 
\[
\Delta\phi+\kappa_{\mathfrak p}^{2}\phi=0, \quad {\bf curl}{\bf curl} \boldsymbol{\psi}-\kappa_{\mathfrak s}^{2}\boldsymbol\psi=0.
\]
In addition, $\phi$ and $\boldsymbol\psi$ are required to satisfy the Sommerfeld radiation condition and the Silver--M\"{u}ller radiation condition, respectively: 
\begin{equation*}
\lim_{\rho\to\infty}\rho(\partial_{\rho}
\phi-\mathrm{i}\kappa_{\mathfrak p}\phi)=0, \quad
\lim_{\rho\to\infty}\rho({\bf curl}\boldsymbol\psi\times\hat{x}-\mathrm{i}\kappa_{\mathfrak s}\boldsymbol\psi)=0
, \quad \rho=|x|.
\end{equation*}

It follows from the Helmholtz decomposition and boundary condition on $\Gamma_D$ that
\[
\boldsymbol{v}=\nabla\phi+{\bf curl}\boldsymbol\psi=-\boldsymbol{u}^i.
\]
Taking the dot product and the cross product of the above equation with $\nu$, respectively, we get
\begin{equation}\label{HelmholtzDec0}
\partial_\nu\phi+\nu\cdot{\bf curl}\boldsymbol\psi=f_1,\quad 
\nu\times\nabla\phi+\nu\times{\bf curl}\boldsymbol\psi=\boldsymbol{f}_2, 
\end{equation}
where 
\[
f_1:=-\nu\cdot\boldsymbol{u}^i, \quad \boldsymbol{f}_2:=-\nu\times\boldsymbol{u}^i.
\]

In summary, the scalar potential function $\phi$ and the vector potential function $\boldsymbol\psi$ satisfy the coupled boundary value problem
\begin{align}\label{HelmholtzDec}
\begin{cases}
\Delta\phi+\kappa_{\mathfrak p}^{2}\phi=0, \quad {\bf curlcurl}\boldsymbol\psi-\kappa_{\mathfrak s}^{2}\boldsymbol\psi=0
\quad &{\rm in} ~\mathbb{R}^3\setminus\overline{D},\\
\partial_\nu\phi+\nu\cdot{\bf curl}\boldsymbol\psi=f_1 \quad 
\nu\times\nabla\phi+\nu\times{\bf curl}\boldsymbol\psi=\boldsymbol{f}_2 \quad &{\rm on} ~ \Gamma_D,\\
\displaystyle{\lim_{\rho\to\infty}\rho(\partial_{\rho}\phi-\mathrm{i}\kappa_{\mathfrak p}\phi)=0}, \quad
\displaystyle{\lim_{\rho\to\infty}\rho({\bf curl}\boldsymbol\psi\times\hat{x}-\mathrm{i}\kappa_{\mathfrak s}\boldsymbol\psi)=0}, \quad &\rho=|x|.
\end{cases}
\end{align}

The following result concerns the uniqueness of the boundary value problem \eqref{HelmholtzDec}.
\begin{theorem}\label{unique1}
	The coupled boundary value problem \eqref{HelmholtzDec} has at most one solution for $\kappa_{\mathfrak p}>0$ and $\kappa_{\mathfrak s}>0$. 
\end{theorem}

\begin{proof}
It suffices to show that $\phi=0$ and $\boldsymbol{\psi}=0$ in $\mathbb{R}^3\setminus\overline{D}$ when $f_1=0, \boldsymbol{f}_2=0$. Let $B_R$ be a ball with radius $R>0$ such that $D\subset B_R$ and $\Gamma_{B} $ be the boundary of $B_R$. Denote by $\Omega$ the bounded domain $\Omega=B_R\setminus\overline{D}$ enclosed by $\Gamma_D$ and $\Gamma_B$. Using the first Green's theorem \cite[$(2.2)$ and $(6.2)$]{DR-book2} and noting $\nabla\cdot\boldsymbol\psi=0$, we have
\begin{align*}
\int_{\Gamma_B}\phi\partial_\nu\bar{\phi}\,\mathrm{d}s&=\int_{\Omega}\big(\phi\Delta\bar{\phi}+\nabla\phi\cdot\nabla\bar{\phi}\big)\,\mathrm{d}x+\int_{\Gamma_D}\phi\partial_\nu\bar{\phi}\,\mathrm{d}s\\
&=\int_{\Omega}\big(-\kappa_{\mathfrak p}^2|\phi|^2+|\nabla\phi|^2\big)\,\mathrm{d}x+\int_{\Gamma_D}\phi\partial_\nu\bar{\phi}\,\mathrm{d}s
\end{align*}
and 
\begin{align*}
\int_{\Gamma_B}\big({\bf curl}\boldsymbol{\bar\psi}\times\hat{x}\big)\cdot\boldsymbol{\psi}\,\mathrm{d}s&=\int_{\Gamma_B}\big(\hat{x}\times\boldsymbol{\psi}\big)\cdot{\bf curl}\boldsymbol{\bar\psi}\,\mathrm{d}s\\ &=\int_{\Omega}\big(\boldsymbol{\psi}\cdot\Delta\boldsymbol{\bar\psi}+{\bf curl}\boldsymbol{\psi}\cdot{\bf curl}\boldsymbol{\bar\psi}\big)\,\mathrm{d}x+\int_{\Gamma_D}\big(\nu\times\boldsymbol{\psi}\big)\cdot{\bf curl}\boldsymbol{\bar\psi}\,\mathrm{d}s\\
&=\int_{\Omega}\big(-\kappa_{\mathfrak s}^2|\boldsymbol{\psi}|^2+|{\bf curl}\boldsymbol{\psi}|^2\big)\,\mathrm{d}x+\int_{\Gamma_D}\big({\bf curl}\boldsymbol{\bar\psi}\times\nu\big)\cdot\boldsymbol{\psi}\,\mathrm{d}s.
\end{align*}

Using the boundary condition \eqref{HelmholtzDec0}, the relation between the gradient and the surface gradient
	\begin{align*}
		\nabla\varphi=\mathbf{Grad}\varphi+\nu \partial_\nu\varphi,
	\end{align*}
	and the identity (cf. \cite[Page 204]{DR-book2})
	\begin{align*}
		\int_{\Gamma_D}\varphi (\nu\cdot{\bf curl}\boldsymbol\psi)\,\mathrm{d}s=\int_{\Gamma_D}\mathbf{Grad}\varphi\cdot(\nu\times\boldsymbol{\psi})\,\mathrm{d}s,
	\end{align*}
	we obtain 
	\begin{align}\label{eqn1}
		\Im\int_{\Gamma_D}\big(\phi\partial_\nu\bar{\phi}+({\bf curl}\boldsymbol{\bar\psi}\times\nu)\cdot\boldsymbol{\psi}\big)\,\mathrm{d}s&=\Im\int_{\Gamma_D}\big(-\phi({\bf curl}\boldsymbol{\bar\psi}\cdot\nu)-(\nabla\bar{\phi}\times\nu)\cdot\boldsymbol{\psi}\big)\,\mathrm{d}s\notag \\
		&=-\Im\int_{\Gamma_D}\big(\mathbf{Grad}\phi\cdot(\nu\times\boldsymbol{\bar\psi})+\nabla\bar{\phi}\cdot(\nu\times\boldsymbol{\psi})\big)\,\mathrm{d}s\notag \\
		&=-\Im\int_{\Gamma_D}\big(\mathbf{Grad}\phi\cdot(\nu\times\boldsymbol{\bar\psi})+\mathbf{Grad}\bar{\phi}\cdot(\nu\times\boldsymbol{\psi})\big)\,\mathrm{d}s\notag \\
		&=0.
	\end{align}
It follows from the radiation conditions \eqref{HelmholtzDec} that
	\begin{equation} \label{eqn2}
		\int_{\Gamma_B}\big(|\partial_\nu\phi|^2+\kappa_{\mathfrak p}^2|\phi|^2+2\kappa_{\mathfrak	p}\Im(\phi\partial_\nu\bar{\phi})\big)\,\mathrm{d}s=\int_{\Gamma_B}|\partial_\nu\phi-\mathrm{i}\kappa_{\mathfrak p}\phi|^2\,\mathrm{d}s\to0
\end{equation}		
and
\begin{equation}\label{eqn3}
		\int_{\Gamma_B}\big(|{\bf curl}\boldsymbol\psi\times\hat{x}|^2 +\kappa_{\mathfrak s}^2|\boldsymbol\psi|^2+2\kappa_{\mathfrak s}\Im(({\bf curl}\boldsymbol{\bar\psi}\times\hat{x})\cdot\boldsymbol{\psi})\big)\,\mathrm{d}s=\int_{\Gamma_B}|{\bf curl}\boldsymbol\psi\times\hat{x}-\mathrm{i}\kappa_{\mathfrak s}\boldsymbol\psi|^2\,\mathrm{d}s\to0
	\end{equation}
	as $R\to\infty$. Since $\kappa_{\mathfrak p}>0$ and $\kappa_{\mathfrak s}>0$, it follows from \eqref{eqn1}--\eqref{eqn3} that
	\begin{align*}
		\lim_{R\to\infty}\int_{\Gamma_B}\Big(\frac{1}{\kappa_{\mathfrak p}}|\partial_\nu\phi|^2+\kappa_{\mathfrak p}|\phi|^2+\frac{1}{\kappa_{\mathfrak s}}|{\bf curl}\boldsymbol\psi\times\hat{x}|^2+\kappa_{\mathfrak s}|\boldsymbol\psi|^2\Big)\,\mathrm{d}s=0.
	\end{align*}
We have from Rellich's lemma that $\phi=0$ and $\boldsymbol{\psi}=0$ in $\mathbb{R}^3\setminus\overline{D}$, which completes the proof.
\end{proof}

It is known that a radiating solution of \eqref{scatteredfield} has the asymptotic behavior of the form
\[
\boldsymbol{v}(x)=\frac{\mathrm{e}^{\mathrm{i}\kappa_{\mathfrak p}|x|}}{|x|}\boldsymbol{v}_{\mathfrak p}^\infty(\hat{x})+\frac{\mathrm{e}^{\mathrm{i}\kappa_{\mathfrak s}|x|}}{|x|}\boldsymbol{v}_{\mathfrak s}^\infty(\hat{x})+\mathcal{O}\left(\frac{1}{|x|^2}\right),
\quad |x|\to\infty
\]
uniformly in all directions $\hat{x}:=x/|x|$, where $\boldsymbol{v}_{\mathfrak p}^\infty$ and $\boldsymbol{v}_{\mathfrak s}^\infty$, defined on
the unit sphere $\mathbb{S}^2=\{\hat{x}\in\mathbb{R}^3: |\hat{x}|=1\}$, are called the compressional and shear far-field patterns of $\boldsymbol{v}$, respectively. 

\begin{remark}
By extending the result \cite[Theorem 3.1]{DLL2019} to three dimensions and using \cite[Theorem 6.9]{DR-book2}, we can establish the relationship between the far-field pattern of the compressional wave $\boldsymbol{v}_p$ or the shear wave $\boldsymbol{v}_s$ and the far-field pattern of the scalar potential $\phi$ or the vector potential $\boldsymbol{\psi}$, i.e., 
\begin{equation}\label{behaviour relation}
	\boldsymbol{v}_{\mathfrak p}^\infty(\hat{x}) =\mathrm{i}\kappa_{\mathfrak p}\phi_{\infty}(\hat{x})\hat{x}, \quad \boldsymbol{v}_{\mathfrak s}^\infty(\hat{x})=\mathrm{i}\kappa_{\mathfrak s}\hat{x}\times\boldsymbol{\psi}_{\infty},
\end{equation}
where the complex-valued
functions $\phi_\infty(\hat{x})$ and $\boldsymbol{\psi}_\infty(\hat{x})$ are the far-field
patterns corresponding to $\phi$ and $\boldsymbol{\psi}$, respectively. Therefore, in view of \eqref{HelmDeco} and \eqref{behaviour relation}, we can obtain the compressional and shear wave components $\boldsymbol{v}_{\mathfrak p}, \boldsymbol{v}_{\mathfrak s}$ and the corresponding far-field patterns $\boldsymbol{v}_{\mathfrak p}^\infty, \boldsymbol{v}_{\mathfrak s} ^\infty$ by solving the coupled boundary value problem \eqref{HelmholtzDec}.
\end{remark}


\section{Boundary integral equations}

In this section, we deduce the coupled system of boundary integral equations for solving the boundary value problem \eqref{HelmholtzDec}. 

Define a vector potential
\begin{align*}
\boldsymbol{A}\boldsymbol{g}(x):=\int_{\Gamma_D}\Phi(x,y;\kappa)\boldsymbol{g}(y)\,\mathrm{d}s(y),  \quad x\in\mathbb{R}^3\setminus\Gamma_D, 
\end{align*}
where $\boldsymbol g$ is a continuous tangential vector function on $\Gamma_D$ and 
\begin{equation}\label{fs3d}
\Phi(x,y;\kappa)=\frac{1}{4\pi}\dfrac{\mathrm{e}^{\mathrm{i}\kappa|x-y|}}{|x-y|}, \quad x\neq y
\end{equation}
is the fundamental solution to the three-dimensional Helmholtz equation. Using \cite[Theorem 6.13]{DR-book2}, we have the jump relation
\begin{align}\label{jumprel}
{\bf curl}\boldsymbol{A}_{\pm}\boldsymbol{g}(x)=\int_{\Gamma_D}\nabla_x\Phi(x,y;\kappa)\times\boldsymbol{g}(y)\,\mathrm{d}s(y)\mp\frac{1}{2}\nu(x)\times\boldsymbol{g}(x),
\end{align}
where
$$
{\bf curl}\boldsymbol{A}_{\pm}\boldsymbol{g}(x):=\lim_{h\to+0}{\bf curl}\boldsymbol{A}\boldsymbol{g}(x\pm h\nu(x)).
$$
Meanwhile, we have from the Maxwell equation that 
 $$({\bf curlcurlcurl}\boldsymbol{A}\boldsymbol{g})_{\pm}(x)=\kappa^2({\bf curl}\boldsymbol{A}\boldsymbol{g})_{\pm}(x).$$

We represent the solutions of \eqref{HelmholtzDec} by  
\begin{align} \label{singlelayer}
\begin{split}
\phi(x)&=\int_{\Gamma_D}\Phi(x,y;\kp)g_1(y)\,\mathrm{d}s(y), \quad x\in\mathbb{R}^3\setminus\Gamma_D\\
\boldsymbol{\psi}(x)&=\frac{1}{\ks^2}{\bf curlcurl} \int_{\Gamma_D}\Phi(x,y;\ks)\boldsymbol{g}_2(y)\,\mathrm{d}s(y),\quad x\in\mathbb{R}^3\setminus\Gamma_D,
\end{split}
\end{align}
where $g_1$ is a scalar density function and $\boldsymbol{g}_2$ is a tangential vector density function satisfying $\boldsymbol{g}_2\cdot\nu=0$. It can be verified from simple calculations that the corresponding far-field patterns can be represented as follows
\begin{align} \label{singlelayer_far}
\phi_\infty(\hat{x})=\frac{1}{4\pi}\int_{\Gamma_D}\mathrm{e}^{-\mathrm{i}\kappa_{\mathfrak p}\hat{x}\cdot y}
g_1(y)\,\mathrm{d}s(y), \quad
\boldsymbol{\psi}_\infty(\hat{x})=\frac{1}{4\pi}\hat{x}\times\int_{\Gamma_D}
\boldsymbol{g}_2(y)\times\hat{x} \mathrm{e}^{-\mathrm{i}\kappa_{\mathfrak s}\hat{x}\cdot y}\,\mathrm{d}s(y). 
\end{align}

Letting $x\in\mathbb{R}^3\setminus\overline{D}$ approach the boundary $\Gamma_D$
in \eqref{singlelayer}, using the jump relations \eqref{jumprel} and 
\begin{align*}
\nabla\phi_\pm(x)=\int_{\Gamma_D}\nabla_x\Phi(x,y;\kappa_{\mathfrak p})g_1(y)\,\mathrm{d}s(y)\mp\frac{1}{2}
\nu(x)g_1(x),
\end{align*}
and the boundary condition \eqref{HelmholtzDec}, we deduce the coupled boundary integral equations for the density functions $g_1$ and $\boldsymbol g_2$ on $\Gamma_D$:  
\begin{align}
\begin{split}\label{boundaryIE0}
f_1(x)&=-\frac{1}{2}g_1(x)+\int_{\Gamma_D}\frac{\partial\Phi(x,y;\kappa_{\mathfrak p})}{\partial\nu(x)}g_1(y)\,\mathrm{d}s(y) 
\\
&\hspace{2cm}+\nu(x)\cdot{\bf curl}_x \int_{\Gamma_D}\Phi(x,y;\kappa_{\mathfrak s})\boldsymbol{g}_2(y)\,\mathrm{d}s(y),
\\
\boldsymbol{f}_2(x)&=\nu(x)\times\nabla_x\int_{\Gamma_D}\Phi(x,y;\kappa_{\mathfrak p}) g_1(y)\,\mathrm{d}s(y)\\
&\hspace{2cm}+\nu(x)\times{\bf curl}_x\int_{\Gamma_D}\Phi(x,y;\kappa_{\mathfrak s})\boldsymbol{g}_2(y)\,\mathrm{d}s(y)  +\frac{1}{2}\boldsymbol{g}_2(x).
\end{split}
\end{align}

Denote by $C(\Gamma_D)$ the space of all continuous functions on $\Gamma_D$, and $\boldsymbol{C}_T(\Gamma_D)$ the space of all continuous tangential vector fields on $\Gamma_D$. The following result concerns the uniqueness of the solution to \eqref{boundaryIE0}.

\begin{theorem}\label{unique2}
The boundary integral equations \eqref{boundaryIE0} has at most one solution in $C(\Gamma_D)\otimes \boldsymbol{C}_T(\Gamma_D)$ provided that $\kappa_{\mathfrak p}$ is not an interior Dirichlet eigenvalue for the Helmholtz equation in $D$ and $\kappa_{\mathfrak s}$ is not an interior Maxwell eigenvalue in $D$ with the homogeneous perfectly conducting boundary condition.
\end{theorem}

\begin{proof}
	It suffices to show that $g_1=0$ and $\boldsymbol{g}_2=0$ if $f_1=0$ and $\boldsymbol{f}_2=0$. By the uniqueness result in Theorem \ref{unique1}, we have
	$$
	\phi(x)=0, \quad \boldsymbol{\psi}(x)=0, \quad x\in\mathbb{R}^3\setminus\overline{D}.
	$$
	It follows from the continuity of the single layer potential that $\phi(x)=0$ for $x\in\Gamma_D$. Since $\kappa_{\mathfrak p}$ is not an interior Dirichlet eigenvalue for the Helmholtz equation in $D$, we get $\phi(x)=0$ for $x\in D$. Using the jump relation of the derivative of the scalar single-layer potential, we obtain $g_1=0$. 
	
	Define the integral operator $\boldsymbol{M}: \boldsymbol{C}_T(\Gamma_D)\to\boldsymbol{C}_T(\Gamma_D)$ by
	\begin{align*}
	(\boldsymbol{M}\boldsymbol{g}_2)(x):=2\nu(x)\times{\bf curl}\int_{\Gamma_D}\Phi(x,y;\kappa_{\mathfrak s})\boldsymbol{g}_2(y)\,\mathrm{d}s(y).
	\end{align*}
	Substituting $g_1=0$ into the second equation of \eqref{boundaryIE0}, we obtain
	$$
	\boldsymbol{g}_2(x)+(\boldsymbol{M}\boldsymbol{g}_2)(x)=0.
	$$
	Since $\kappa_{\mathfrak s}$ is not an interior Maxwell eigenvalue in $D$ with the perfectly conducting boundary condition $\nu\times\boldsymbol{\psi}=0$,  we obtain from \cite[Theorem 4.23]{DR-book1} that $\boldsymbol{g}_2=0$, which completes the proof.
\end{proof}

Next we introduce the single-layer boundary integral operators $S^\sigma, \sigma=\mathfrak{p,s}$ and the normal derivative boundary integral operator $K$ for $g\in C(\Gamma_{D})$ and $\bg \in\boldsymbol{C}_T(\Gamma_{D})$ by
\begin{align*}
	(S^\mathfrak{p} g)(x)&=2\int_{\Gamma_D} \Phi(x,y;\kappa_\mathfrak{p})g(y)\,\mathrm{d}s(y), \quad x\in\Gamma_D, 
	\\
	(S^\mathfrak{s} \bg)(x)&=2\int_{\Gamma_D} \Phi(x,y;\kappa_\mathfrak{s})\bg(y)\,\mathrm{d}s(y), \quad x\in\Gamma_D, 
	\\
	(Kg)(x)&=2\int_{\Gamma_D}\frac{\partial\Phi(x,y;\kappa_\mathfrak{p})}
	{\partial\nu(x)}g(y)\,\mathrm{d}s(y), \quad x\in\Gamma_D,
\end{align*}
and we also define boundary integral operators $N$, $H$  and $M$ by
\begin{align*}
(N\bg)(x)&=2\nu(x)\cdot{\bf curl}_x \int_{\Gamma_D}\Phi(x,y;\kappa_{\mathfrak s})\bg(y)\,\mathrm{d}s(y),\\
(Hg)(x)&=2\nu(x)\times\nabla_x\int_{\Gamma_D}\Phi(x,y;\kappa_{\mathfrak p}) g(y)\,\mathrm{d}s(y),\\
(M\bg)(x)&=2\nu(x)\times{\bf curl}_x\int_{\Gamma_D}\Phi(x,y;\kappa_{\mathfrak s})\bg(y)\,\mathrm{d}s(y),
\end{align*}
where $g$ and $\bg$ are the scalar and vector densities, respectively. Then the system \eqref{boundaryIE0} can be rewritten as
\begin{align} \label{boundaryoperator}
	\begin{cases}
		-g_1+Kg_1+N\boldsymbol{g}_2=f_1,\\
		Hg_1+\boldsymbol{g}_2+M\boldsymbol{g}_2=\boldsymbol{f}_2.
	\end{cases}
\end{align}

By the decomposition in \cite[$(2.5)-(2.6)$]{GG2004}, the kernels of weakly singular integral operators $S^\sigma$ and $K$ can be decomposed into a general form as
\begin{align*}
	m(x,y)=\frac{1}{|x-y|}m_1(x,y)+m_2(x,y)
\end{align*}
with $m_i$ $(i=1, 2)$ given by
\begin{align*}
	m_i(x,y)=m_{i,1}(x,y)+m_{i,2}(x,y)\frac{(x-y)\cdot\nu(y)}{|x-y|^2}+m_{i,3}(x,y)\frac{(x-y)\cdot\nu(x)}{|x-y|^2},
\end{align*}
where each $m_{i,j} (i=1,2, j=1,2,3)$ is infinitely continuously differentiable on $\mathbb{R}^3\times\mathbb{R}^3$ under the assumption that $\Gamma_D$ is analytic. Hence, we have for $x\in\Gamma_D$ that
\begin{align} \label{operatorsD}
	\begin{split}
		(S^\mathfrak{p} g)(x)&=\int_{\Gamma_D}
		\Big(\frac{1}{|x-y|}s^{\mathfrak{p}}_{1}(x,y)+s^{\mathfrak{p}}_{2}(x,y)\Big)g(y)\,\mathrm{d}s(y),\\
		(S^\mathfrak{s} \bg)(x)&=\int_{\Gamma_D}
		\Big(\frac{1}{|x-y|}s^{\mathfrak{s}}_{1}(x,y)+s^{\mathfrak{s}}_{2}(x,y)\Big)\bg(y)\,\mathrm{d}s(y),\\
		(Kg)(x)&=\int_{\Gamma_D}\Big(\frac{1}{|x-y|}k^\mathfrak{p}_{1}(x,y)+k^\mathfrak{p}_{2}(x,y)\Big)g(y)\,\mathrm{d}s(y),
	\end{split}
\end{align}
where
\begin{align*}
	s^\sigma_{1}(x,y)&=\frac{\cos\big(\kappa_{\sigma}|x-y|\big)}{2\pi}, \quad s^\sigma_{2}(x,y)=\begin{cases}
		\frac{\mathrm{i}\sin\big(\kappa_{\sigma}|x-y|\big)}{2\pi|x-y|}, &x\not=y,\\
		\frac{\mathrm{i}\kappa_{\mathfrak{\sigma}}}{2\pi},  &x=y,
	\end{cases},\quad  \sigma=\mathfrak{p}, \mathfrak{s},  \\
	k^{\mathfrak{p}}_{1}(x,y)&=-\frac{\nu(x)\cdot(x-y)}{|x-y|^2}s^{\mathfrak{p}}_{1}(x,y)+\mathrm{i}\kappa_{\mathfrak{p}}\nu(x)\cdot(x-y)s^{\mathfrak{p}}_{2}(x,y), \\
	k^{\mathfrak{p}}_{2}(x,y)&=-\frac{\nu(x)\cdot(x-y)}{|x-y|^2}\big[s^{\mathfrak{p}}_{2}(x,y)-\mathrm{i}\kappa_{\mathfrak{p}} s^{\mathfrak{p}}_{1}(x,y)\big].
\end{align*}

Similarly, the kernel of weakly singular integral operator $M$ can be decomposed into 
\begin{align*}
(M\boldsymbol{g})(x)&=\int_{\Gamma_D}\Big(\frac{1}{|x-y|}\boldsymbol{m}_1(x,y)+\boldsymbol{m}_2(x,y)\Big)\boldsymbol{g}(y)\,\mathrm{d}s(y)
\end{align*}
with $\boldsymbol{m}_i$ $(i=1, 2)$ being of the form
\begin{align*}
	\boldsymbol{m}_i(x,y)=\widetilde{m}_{i,1}(x,y)\frac{(x-y)[\nu(x)-\nu(y)]^\top}{|x-y|^2}+\widetilde{m}_{i,2}(x,y)\frac{(x-y)\cdot\nu(y)}{|x-y|^2}I+\widetilde{m}_{i,3}(x,y),
\end{align*}
where $I$ is the $3\times3$ identity matrix, $\widetilde{m}_{2,3}$ is the $3\times3$ zero matrix, and each $\widetilde{m}_{i,j} (i=1,2, j=1,2,3)$ is infinitely continuously differentiable on $\mathbb{R}^3\times\mathbb{R}^3$.
We refer to \cite[$(2.8)-(2.9)$]{GH2008} for the detailed expressions of $\widetilde{m}_{i,j}$.

Since all the boundary operators $S^{\sigma}$, $K$, $M$ have weakly singular kernels, their spectrally accurate discretization can be conveniently obtained by using the spherical harmonics. The details are given in Section \ref{SM}. However, for the integral operators $N$ and $H$, their kernels have Cauchy type strong singularities, direct discretization will introduce large errors. In order to design a spectral method for \eqref{boundaryoperator} via the Galerkin discretization, it is necessary to regularize the singularity of $N$ and $H$, which is examined in the next section.

\section{Spherical parametrization}

From now on, we assume that the boundary $\Gamma_{D}$ is an isomorphism of a unit sphere, which is a common assumption in the areas of  wave scattering and inverse scattering\cite{DR-book2}. For two vectors $\boldsymbol{a},\boldsymbol{b}\in\mathbb{R}^3$ and two $3\times3$ matrices $A=(a_1,a_2,a_3), B=(b_1,b_2,b_3)$, we define
$$
\boldsymbol{a}\otimes\boldsymbol{b}=\boldsymbol{a}\boldsymbol{b}^\top,\quad
A:B=a_1\cdot b_1 + a_2\cdot b_2 + a_3\cdot b_3.
$$

 Based on a bijective parameterization map $\boldsymbol{q}: \mathbb{S}^2\to\Gamma_D$, the boundary integral equations \eqref{boundaryoperator} can be transformed on the unit sphere $\mathbb{S}^2$. By the change of variables $x=\boldsymbol{q}(\hat{x})$, for any integrable function $g$ defined on $\Gamma_D$, it holds
\begin{align*}
	\int_{\Gamma_D}g(x)\,\mathrm{d}s(x)=\int_{\mathbb{S}^2}g(\boldsymbol{q}(\hat{x}))J_{\boldsymbol{q}}(\hat{x})\,\mathrm{d}s(\hat{x}),
\end{align*}
where $J_{\boldsymbol{q}}(\hat{x})$ is the Jacobian of the transformation $\boldsymbol{q}$.

Denote by $(\theta, \varphi)$ the spherical coordinates of any point $\hat{x}\in\mathbb{S}^2$, i.e.,
\begin{equation*}
\hat{x}=p(\theta,\varphi)=(\sin\theta\cos\varphi, \sin\theta\sin\varphi, \cos\theta)^\top, \quad\theta\in[0, \pi], ~\varphi\in[0, 2\pi),
\end{equation*}
and the corresponding Jacobian is $J_p(\theta,\varphi)=\sin\theta$.
The tangent plane at any point $\hat{x}\in\mathbb{S}^2$ are generated by the unit vectors
\begin{align*}
	\boldsymbol{e}_\theta\circ p&=\frac{\partial p}{\partial\theta}(\theta,\varphi)=(\cos\theta\cos\varphi, \cos\theta\sin\varphi, -\sin\theta)^\top,\\
	\boldsymbol{e}_\varphi\circ p&=\frac{1}{\sin\theta}\frac{\partial p}{\partial\varphi}(\theta,\varphi)=(-\sin\varphi, \cos\varphi, 0)^\top.
\end{align*}
The triplet $(\hat{x},\boldsymbol{e}_\theta,\boldsymbol{e}_\varphi)$ forms an orthonormal system. Following the notations in \cite{Louer2014,Louer2020}, we define the tangent vectors on $\Gamma_{D}$ by
\begin{align*}
	\boldsymbol{t}_1(\hat{x})&=\frac{\partial\boldsymbol{q}\circ p}{\partial\theta}\circ p^{-1}(\hat{x})=[{\bf D}_{\mathbb S^2}\bq(\hat{x})]\boldsymbol{e}_\theta(\hat{x}),\\
	\boldsymbol{t}_2(\hat{x})&=\Big(\frac{1}{\sin\theta}\frac{\partial\boldsymbol{q}\circ p}{\partial\varphi}\Big)\circ p^{-1}(\hat{x})=[{\bf D}_{\mathbb S^2}\bq(\hat{x})]\boldsymbol{e}_\varphi(\hat{x}),
\end{align*}
where the matrix $[{\bf D}_{\mathbb{S}^2}\boldsymbol{q}(\hat{x})]=\boldsymbol{t}_1\otimes\boldsymbol{e}_\theta+\boldsymbol{t}_2\otimes\boldsymbol{e}_\varphi$ maps the tangent plane to $\mathbb S^2$ at the point $\hat x$ onto the tangent plane to $\Gamma_D$ at the point $\boldsymbol q(\hat x)$, the Jacobian $J_{\boldsymbol{q}}$ and the normal vector $\nu\circ \boldsymbol{q}$ are given by 
\begin{align*}
J_{\bq}=|\boldsymbol{t}_1\times\boldsymbol{t}_2|, \quad \nu\circ \boldsymbol{q}=\frac{\boldsymbol{t}_1\times\boldsymbol{t}_2}{J_{\bq}}.
\end{align*}

By the conclusions in \cite{Louer2014,Louer2020} and the change of variables $x=\boldsymbol{q}(\hat{x})$, for any smooth $w$ and $\boldsymbol{w}=(w_1,w_2,w_3)^\top$, we may define the following surface scalar and vector curl operators on $\Gamma_D$: 
\begin{align}
	({\bf curl}_{\Gamma_D}w)\circ\bq &= \frac{1}{J_{\bq}}[{\bf D}_{\mathbb S^2}\bq(\hat{x})]{\bf curl}_{\mathbb{S}^2}(w\circ\bq)\label{vectorcurl},\\
	({\rm curl}_{\Gamma_D}\bw)\circ\bq &=-\frac{1}{J_{\bq}}[{\bf D}_{\mathbb{S}^2}\bq(\hat{x})]^\top:{\bf curl}_{\mathbb{S}^2}(\bw\circ\bq),\label{scalarcurl}
\end{align}
where ${\bf curl}_{\mathbb{S}^2}(w\circ\bq) =\big({\mathbf{Grad}_{\mathbb{S}^2}} (w\circ\bq)\big) \times \hat{x}$, and
\begin{align}
	{\mathbf{Grad}_{\mathbb{S}^2}}( w\circ\bq) &= \Big(\frac{\partial w\circ\boldsymbol{q}\circ p}{\partial\theta}\Big)\circ p^{-1}\boldsymbol{e}_\theta+\Big(\frac{1}{\sin\theta}\frac{\partial w\circ\boldsymbol{q}\circ p}{\partial\varphi}\Big)\circ p^{-1}\boldsymbol{e}_\varphi. \label{PGrad2}
\end{align}
Here ${\bf curl}_{\mathbb{S}^2} (\boldsymbol{w}\circ\boldsymbol{q})$ is a matrix whose  $j$-th column is ${\bf curl}_{\mathbb{S}^2} w_j\circ\boldsymbol{q}$. Using \eqref{PGrad2} and the representation of ${\bf D}_{\mathbb{S}^2}\bq$ and ${\bf curl}_{\mathbb{S}^2}(\bw\circ\bq)$, we obtain 
\begin{align}\label{curls2}
	\big([{\bf D}_{\mathbb{S}^2}\bq]^\top:{\bf curl}_{\mathbb{S}^2}(\bw\circ\bq)\big)\circ p=\frac{1}{\sin\theta}(\boldsymbol{t}_1\circ p)\cdot\frac{\partial\boldsymbol{w}\circ\boldsymbol{q}\circ p}{\partial\varphi}-(\boldsymbol{t}_2\circ p)\cdot\frac{\partial\boldsymbol{w}\circ\boldsymbol{q}\circ p}{\partial\theta}.
\end{align}

After the parametrization, the boundary operators $S^{\sigma}$, $K$, $M$ can be equivalently rewritten as:
\begin{align} \label{operatorsB}
	\begin{split}
		(\mathcal{S}^{\mathfrak p} G)(\hat{x})&=\int_{\mathbb{S}^2}\Big(\frac{1}{|\hat{x}-\hat{y}|}\widetilde{S}^{\mathfrak p}_{1}(\hat{x},\hat{y})+\widetilde{S}^{\mathfrak p}_{2}(\hat{x},\hat{y})\Big)G(\hat{y})\,\mathrm{d}s(\hat{y}),\\
		(\mathcal{S}^{\mathfrak s}  \boldsymbol{G})(\hat{x})&=\int_{\mathbb{S}^2}\Big(\frac{1}{|\hat{x}-\hat{y}|}\widetilde{S}^{\mathfrak s}_{1}(\hat{x},\hat{y})+\widetilde{S}^{\mathfrak s}_{2}(\hat{x},\hat{y})\Big) \boldsymbol{G}(\hat{y})\,\mathrm{d}s(\hat{y}),\\
		(\mathcal{K}G)(\hat{x})&=\int_{\mathbb{S}^2}\Big(\frac{1}{|\hat{x}-\hat{y}|}\widetilde{K}_{1}(\hat{x},\hat{y})+\widetilde{K}_{2}(\hat{x},\hat{y})\Big)G(\hat{y})\,\mathrm{d}s(\hat{y}),
		\\
		(\mathcal{M}\boldsymbol{G})(\hat{x})&=\int_{\mathbb{S}^2}\Big(\frac{1}{|\hat{x}-\hat{y}|}\widetilde{M}_{1}(\hat{x},\hat{y})+\widetilde{M}_{2}(\hat{x},\hat{y})\Big)\boldsymbol{G}(\hat{y})\,\mathrm{d}s(\hat{y}),
	\end{split}
\end{align}
where $G=g\circ\boldsymbol{q},~ \boldsymbol{G}=\boldsymbol{g}\circ\boldsymbol{q}$, and the kernels for $\sigma={\mathfrak p}$ or ${\mathfrak s}$ are given by
\begin{align*}
	&R(\hat{x},\hat{y})=\frac{|\hat{x}-\hat{y}|}{|\boldsymbol{q}(\hat{x})-\boldsymbol{q}(\hat{y})|},\\
	&\widetilde{S}^{\sigma}_{1}(\hat{x},\hat{y})=R(\hat{x},\hat{y})s^\sigma_{1}(\boldsymbol{q}(\hat{x}),\boldsymbol{q}(\hat{y}))J_{\boldsymbol q}(\hat{y}), 
	&&\widetilde{S}^{\sigma}_{2}(\hat{x},\hat{y})=s^\sigma_{2}(\boldsymbol{q}(\hat{x}),\boldsymbol{q}(\hat{y}))J_{\boldsymbol q}(\hat{y}), \\
	&\widetilde{K}_{1}(\hat{x},\hat{y})=R(\hat{x},\hat{y})k_{1}^{\mathfrak{p}}(\boldsymbol{q}(\hat{x}),\boldsymbol{q}(\hat{y}))J_{\boldsymbol q}(\hat{y})J_{\boldsymbol q}(\hat{x}), 
	&&\widetilde{K}_{2}(\hat{x},\hat{y})=k_{2}^{\mathfrak{p}}(\boldsymbol{q}(\hat{x}),\boldsymbol{q}(\hat{y}))J_{\boldsymbol q}(\hat{y})J_{\boldsymbol q}(\hat{x}),\\
	&\widetilde{M}_{1}(\hat{x},\hat{y})=R(\hat{x},\hat{y}){\boldsymbol m}_1(\boldsymbol{q}(\hat{x}),\boldsymbol{q}(\hat{y}))J_{\boldsymbol q}(\hat{y})J_{\boldsymbol q}(\hat{x}),  &&\widetilde{M}_{2}(\hat{x},\hat{y})={\boldsymbol m}_2(\boldsymbol{q}(\hat{x}),\boldsymbol{q}(\hat{y}))J_{\boldsymbol q}(\hat{y})J_{\boldsymbol q}(\hat{x}).
\end{align*}

As we mentioned before, due to the strong singularities in $N$ and $H$, it is difficult to achieve high order accuracy by 
the direct discretization of $N$ and $H$. As a regularization technique, the following two theorems describe the Galerkin  approach by transforming the singularities of $N$ and $H$ to the test functions. 
 
\begin{theorem}\label{ThN}
For any smooth scalar function $\varphi(\hat{x})$ on $\mathbb{S}^2$, we have
\begin{align}\label{JqNg}
\big(J_{\boldsymbol q} (N\boldsymbol{g})\circ\boldsymbol{q},\varphi\big)=\big(\mathcal{S}^\mathfrak{s}\boldsymbol{G},[{\bf D}_{\mathbb S^2}\boldsymbol{q}]{\bf curl}_{\mathbb{S}^2}\varphi\big),
\end{align}
where $(\cdot,\cdot)$ is the $L^2$ inner product on ${\mathbb S^2}$. 
\end{theorem}

\begin{proof}	
Using \cite[Theorem 2.5.20]{Nedelec2001}, we have
\begin{align} \label{Ng}
N\boldsymbol{g}(x)={\rm curl}_{\Gamma_D}\big( (\nu(x)\times (S^{\mathfrak s}\boldsymbol{g})(x))\times\nu(x)\big),\quad x\in\Gamma_D.
\end{align}
It follows from Stokes' theorem that 
\begin{align}\label{intebyparts}
\int_{\Gamma_D}\boldsymbol{w}\cdot{\bf curl}_{\Gamma_D}u\,\mathrm{d}s=\int_{\Gamma_D}u~{\rm curl}_{\Gamma_D} \boldsymbol{w}\,\mathrm{d}s,
\end{align}
for any differentiable scalar function $u$ and tangential vector function $\boldsymbol{w}$ on $\Gamma_{D}$.
Combining \eqref{vectorcurl} and \eqref{Ng}--\eqref{intebyparts} gives
\begin{align*}
\big(J_{\boldsymbol q}(N\boldsymbol{g})\circ\boldsymbol{q},\varphi\big) &=\int_{\mathbb{S}^2}J_{\boldsymbol q}(\hat{x})(N\boldsymbol{g})\circ\boldsymbol{q}(\hat{x})\overline{\varphi(\hat{x})}\,\mathrm{d}s(\hat{x})\\
&=\int_{\Gamma_D}(N\boldsymbol{g})(x)\overline{\varphi\circ\boldsymbol{q}^{-1}(x)}\,\mathrm{d}s(x)\\
&=\int_{\Gamma_D}\overline{\varphi\circ\boldsymbol{q}^{-1}(x)}~{\rm curl}_{\Gamma_D}\big((\nu(x)\times(S^{\mathfrak s}\boldsymbol{g})(x))\times\nu(x)\big)\,\mathrm{d}s(x)\\
&=\int_{\Gamma_D}\big((\nu(x)\times(S^{\mathfrak s}\boldsymbol{g})(x))\times\nu(x)\big)\cdot\overline{{\bf curl}_{\Gamma_D}\varphi\circ\boldsymbol{q}^{-1}(x)}\,\mathrm{d}s(x)\\
&=\int_{\Gamma_D}(S^{\mathfrak s}\boldsymbol{g})(x)\cdot\overline{{\bf curl}_{\Gamma_D}\varphi\circ\boldsymbol{q}^{-1}(x)}\,\mathrm{d}s(x)\\
&=\int_{\mathbb{S}^2}(\mathcal{S}^{\mathfrak s}\boldsymbol{G})(\hat{x})\cdot\overline{[{\bf D}_{\mathbb S^2}\boldsymbol{q}(\hat{x})]{\bf curl}_{\mathbb{S}^2}\varphi(\hat{x})}\,\mathrm{d}s(\hat{x}),
\end{align*}
which completes the proof.
\end{proof}

\begin{theorem}\label{ThH}
	For any smooth vector function $\boldsymbol\varphi(\hat{x})$ on $\mathbb{S}^2$, we have
	\begin{align}\label{JqHg}
		\big(J_{\boldsymbol q}(Hg)\circ\boldsymbol{q},\boldsymbol{\varphi}\big)=\big(\mathcal{S}^\mathfrak{p}G,[{\bf D}_{\mathbb{S}^2}\bq]^\top:{\bf curl}_{\mathbb{S}^2}\boldsymbol\varphi\big).
	\end{align}
\end{theorem}

\begin{proof}
	For $x\in\Gamma_D$, since $\nabla w=\mathbf{Grad}_{\Gamma_D}w+\nu\partial_\nu w$, we have
	\begin{align} \label{Hg}
		(Hg)(x)=2\nu(x)\times\mathbf{Grad}_{\Gamma_D}\int_{\Gamma_D}\Phi(x,y;\kp) g(y)\,\mathrm{d}s(y)=-\mathbf{curl}_{\Gamma_D}(S^\mathfrak{p}g)(x).
	\end{align}
	Using \eqref{scalarcurl}, \eqref{Hg} and stokes' theorem \eqref{intebyparts}, 
	we obtain 
	\begin{align*}
		\big(J_{\boldsymbol q}(Hg)\circ\bq,\boldsymbol\varphi\big) &=\int_{\mathbb{S}^2}J_{\boldsymbol q}(\hat{x})[(Hg)\circ\boldsymbol{q}(\hat{x})]\cdot\overline{\boldsymbol\varphi(\hat{x})}\,\mathrm{d}s(\hat{x})\\
		&=\int_{\Gamma_D}(Hg)(x)\cdot\overline{\boldsymbol\varphi\circ\boldsymbol{q}^{-1}(x)}\,\mathrm{d}s(x)\\
		&=-\int_{\Gamma_D}\mathbf{curl}_{\Gamma_D}(S^\mathfrak{p}g)(x)\cdot\overline{\boldsymbol\varphi\circ\boldsymbol{q}^{-1}(x)}\,\mathrm{d}s(x)\\
		&=-\int_{\Gamma_D}(S^\mathfrak{p}g)(x)\overline{{\rm curl}_{\Gamma_D}\boldsymbol\varphi\circ\boldsymbol{q}^{-1}(x)}\,\mathrm{d}s(x)\\
		&=\int_{\mathbb{S}^2}(\mathcal{S}^{\mathfrak p}G)(\hat{x})\overline{[{\bf D}_{\mathbb{S}^2}\bq]^\top:{\bf curl}_{\mathbb{S}^2}\boldsymbol\varphi}\,\mathrm{d}s(\hat{x}),
	\end{align*}
	which completes the proof.
\end{proof}

\section{Numerical discretization}\label{SM}

Motivated by \cite{GG2004}, we propose a fully discrete Galerkin type method  with spectral accuracy. To approximate the scalar density functions on the unit sphere, we choose $(n+1)^2$-dimensional space of all spherical harmonics of degree less than or equal to $n$, denoted by
\begin{align*}
	X_n=\text{span}\{Y_{l,j}(\hat{x}): 0\leq l\leq n,~|j|\leq l\},
\end{align*}
where 
\begin{align*}
	Y_{l,j}(\hat{x})=Y_{l,j}(p(\theta,\varphi))=c_l^jP_l^{|j|}(\cos\theta)\mathrm{e}^{\mathrm{i}j\varphi}, \quad c_l^j=(-1)^{(j+|j|)/2}\sqrt{\frac{2l+1}{4\pi}\frac{(l-|j|)!}{(l+|j|)!}}
\end{align*}
for $l=0,1,2,\cdots, |j|\leq l$ form a complete orthonormal system in $L^2(\mathbb{S}^2)$, and $P_l^{|j|}$ denote the associated Legendre functions of degree $l$ with order $|j|$. Analogously to \cite{GH2007}, we introduce 
\begin{align*}
	\boldsymbol{X_n}=\text{span}\{\boldsymbol{Y}_{l,j,k}(\hat{x}):  \boldsymbol{Y}_{l,j,k}=Y_{l,j}\boldsymbol{e}_k,~0\leq l\leq n,~|j|\leq l,k=1,2,3\},
\end{align*}
where $\boldsymbol{e}_k$ denotes the $k$th Euclidean vector.

It follows from \cite{Louer2020} and \eqref{PGrad2} that the tangential gradient of the spherical harmonics is given by
\begin{align*}
	\mathbf{Grad}_{\mathbb{S}^2}Y_{l,j}(p(\theta,\varphi))=
	\begin{cases}
		c_l^j\Big(\frac{\partial P_l^{|j|}(\cos\theta)}{\partial\theta}\mathrm{e}^{\mathrm{i}j\varphi}\boldsymbol{e}_\theta\circ p+\mathrm{i}jP_l^{|j|}(\cos\theta)\mathrm{e}^{\mathrm{i}j\varphi}\frac{\boldsymbol{e}_\varphi\circ p}{\sin\theta}\Big),  &\sin\theta\not=0,\\
		\sqrt{\frac{2l+1}{4\pi}}\sqrt{l(l+1)}\Big(\frac{(\cos\theta)^l}{2}\boldsymbol{e}_\theta\circ p+\mathrm{i}j\frac{(\cos\theta)^{l+1}}{2}\boldsymbol{e}_\varphi\circ p\Big), &\sin\theta=0, |j|=1,\\
		(0,0,0)^\top, &\sin\theta=0, |j|\not=1,
	\end{cases}
\end{align*}
where 
\begin{align*}
	\frac{\partial P_l^{|j|}(\cos\theta)}{\partial\theta}=
	-\frac{(l+1)\cos\theta}{\sin\theta}P_l^{|j|}(\cos\theta)+\frac{l-|j|+1}{\sin\theta}P_{l+1}^{|j|}(\cos\theta), \quad \sin\theta\not=0.
\end{align*}

It is clear to note that $\mathbf{Grad}_{\mathbb{S}^2}Y_{l,j}$ is a tangential vector on $\mathbb{S}^2$ but may not be a tangential vector on the boundary $\Gamma_{D}$. To approximate the tangential vector density functions on the parametrized surface $\Gamma_D$, we choose the following ansatz space \cite{GH2008}:
\begin{align*}
	\underline{\mathbb{T}_n}=\text{span}\Big\{\boldsymbol{Z}_{l,j}^{(\tilde{k})}(\hat{x}): 1\leq l\leq n, ~|j|\leq l, ~\tilde{k}=1,2\Big\},
\end{align*}
where
\begin{align*}
	\boldsymbol{Z}_{l,j}^{(1)}(\hat{x}) &= \frac{1}{\sqrt{l(l+1)}}\mathcal{F}(\hat{x}) \mathbf{Grad}_{\mathbb{S}^2}Y_{l,j}(\hat{x}), \\
	\boldsymbol{Z}_{l,j}^{(2)}(\hat{x}) &=  \frac{1}{\sqrt{l(l+1)}} \mathcal{F}(\hat{x}) \hat{x}\times \mathbf{Grad}_{\mathbb{S}^2}Y_{l,j}(\hat{x}),
\end{align*}
and $\boldsymbol{Z}_{0,0}^{(1)}=\boldsymbol{Z}_{0,0}^{(2)}=0$. Here $\mathcal{F}(\hat{x}) $ is an orthogonal transformation that transforms tangential functions on  $\mathbb{S}^2$ to tangential functions on $\Gamma_D$. More explicitly, for a given vector $\mathbf{y}\in \mathbb{C}^3$, $\mathcal{F}(\hat{x})$ is given by
\begin{eqnarray*}
	\mathcal{F}(\hat{x}) \mathbf{y}= \cos\psi \mathbf{y} + [\hat{x}\times \nu \circ \bq] \times \mathbf{y}+\frac{1}{1+\cos\psi} [\hat{x}\times \nu \circ \bq]^{\top} \mathbf{y}[\hat{x}\times \nu \circ \bq],
\end{eqnarray*}
where $\psi$ is the angle between $\hat{x}$ and $\nu \circ \bq(\hat{x})$. More properties on $\mathcal{F}(\hat{x})$ can be found in \cite{GH2008}. 

Let $-1<z_1<z_2<\cdots<z_{n+1}<1$ denote the zeros of the Legendre polynomial $P_{n+1}$, and consider the Gaussian product rule 
for the numerical integration of a continuous function over $\mathbb{S}^2$: 
\begin{align*}
	\int_{\mathbb{S}^2} f(\hat{y})\,\mathrm{d}s(\hat{y})\approx\sum_{r=0}^{2n+1}\sum_{s=1}^{n+1}\mu_r\nu_sf(p(\theta_{s},\varphi_{r})):=Q_n(f), 
\end{align*}
where the weights $\mu_r$ and $\nu_s$ are given by 
$$
\mu_r=\frac{\pi}{n+1}, \quad\nu_{s}=\frac{2(1-z_{s}^2)}{[(n+1)P_{n}(z_{s})]^2}, 
$$  
and the quadrature knots $\theta_s$ and $\varphi_r$ are 
$$
\theta_{s}=\arccos z_{s}, \qquad \varphi_{r}=\frac{r\pi}{n+1}.
$$ 

Let $C(\mathbb{S}^2)$ be the space of continuous functions on  $\mathbb{S}^2$ and $ \boldsymbol{C}(\mathbb{S}^2)$ the vector function space on $\mathbb{S}^2$ with each component in $C(\mathbb{S}^2)$. Define the discrete orthogonal projection operators $\mathcal{L}_n^{\mathfrak{p}}: C(\mathbb{S}^2)\to X_n$ and $\mathcal{L}_n^{\mathfrak{s}}: \boldsymbol{C}(\mathbb{S}^2)\to \boldsymbol{X_n}$ by
\begin{align*}
	\mathcal{L}_n^{\mathfrak{p}}\psi&=\sum_{l=0}^n\sum_{|j|\leq l}(\psi,Y_{l,j})_nY_{l,j}, \quad \psi\in C(\mathbb{S}^2),\\
	\mathcal{L}_n^{\mathfrak{s}}\boldsymbol{\Psi}&=\sum_{l=0}^n\sum_{|j|\leq l}\sum_{k=1}^3(\boldsymbol{\Psi},\boldsymbol{Y}_{l,j,k})_n\boldsymbol{Y}_{l,j,k}, \quad \Psi\in \boldsymbol{C}(\mathbb{S}^2),	
\end{align*}
where we have set $(\psi,Y_{l,j})_n=Q_n(\psi \overline{Y}_{l,j})$ and
the discrete inner product on $\mathbb{S}^2$ for two vector functions $\boldsymbol{G}$ and $\boldsymbol{Z}$ is denoted by $(\boldsymbol{G},\boldsymbol{Z})_{ n}=Q_n( \boldsymbol{Z}^\top\overline{\boldsymbol{G}})$.

Now we describe the approximation in details for \eqref{JqNg} and \eqref{JqHg}. Following \cite{GG2004,GH2008}, we split 
the kernels in \eqref{operatorsB} into a weakly singular part and an analytic part. In order to accurately integrate the weakly singular part, an orthogonal transformation is introduced on $\mathbb{S}^2$ and the singularity is transferred to the north pole $\hat{n}=(0,0,1)$. In particular, if $\hat{x}=p(\theta,\varphi)$, the orthogonal transformation is defined by $T_{\hat{x}}:=D_P(\varphi)D_Q(\theta)D_P(-\varphi)$, i.e., $T_{\hat{x}}\hat{x}=\hat{n}$, where 
$$
D_P(\psi)=
\left[ \begin{array}{ccc}
\cos\psi & -\sin\psi & 0\\
\sin\psi & \cos\psi  & 0\\
0     &     0     & 1
\end{array} 
\right ], \quad
D_Q(\psi)=
\left[ \begin{array}{ccc}
\cos\psi &   0   & -\sin\psi \\
0     &   1   &    0      \\
\sin\psi &   0   & \cos\psi  
\end{array} 
\right ].
$$
Define the linear and bilinear transformations by
\begin{align*}
	T_{\hat{x}}\Psi(\hat{z}):=\Psi(T_{\hat{x}}^{-1}\hat{z}),\quad T_{\hat{x}}\Psi(\hat{z}_1,\hat{z}_2):=\Psi(T_{\hat{x}}^{-1}\hat{z}_1,T_{\hat{x}}^{-1}\hat{z}_2).
\end{align*}
Using the fact that 
\begin{eqnarray*}
 |\hat{x}-\hat{y}|=|T_{\hat{x}}^{-1}(\hat{n}-\hat{z})|=|\hat{n}-\hat{z}|,
\end{eqnarray*}
we can write $\mathcal{S}^\sigma$ in \eqref{operatorsB} as  
\begin{align*}
(\mathcal{S}^\sigma \widetilde{G}_\sigma)(\hat{x})=\int_{\mathbb{S}^2}\Big(\frac{1}{|\hat{n}-\hat{z}|}T_{\hat{x}}\widetilde{S}^{\sigma}_{1}(\hat{n},\hat{z})+T_{\hat{x}}\widetilde{S}^{\sigma}_{2}(\hat{n},\hat{z})\Big)
T_{\hat{x}}\widetilde{G}_\sigma(\hat{z})\,\mathrm{d}s(\hat{z})
\end{align*}
for $\sigma=\mathfrak{p,s}$, where $\widetilde{G}_{\mathfrak{p}}=g_1\circ \bq$ and $\widetilde{G}_{\mathfrak{s}}=\boldsymbol{g}_2\circ\bq$.
Then, by using
$$
\int_{\mathbb{S}^2}\frac{1}{|\hat{x}-\hat{y}|}Y_{l,j}(\hat{y})\,\mathrm{d}s(\hat{y})=\frac{4\pi}{2l+1}Y_{l,j}(\hat{x}),\quad \hat{x}\in\mathbb{S}^2
$$ 
and the addition theorem
$$
\sum_{j=-l}^{l}Y_{l,j}(\hat{x})\overline{Y_{l,j}(\hat{y})}=\frac{2l+1}{4\pi}P_l(\cos\bar{\theta}),
$$
where $\bar{\theta}$ denotes the angle between $\hat{x}$ and $\hat{y}$, the approximation $\mathcal{S}^\sigma_{n'}$ for the operators $\mathcal{S}^\sigma$ can be represented as 
\begin{align*}
(\mathcal{S}^\sigma_{n'}\widetilde{G}_\sigma)(\hat{x}):=&\int_{\mathbb{S}^2}\frac{1}{|\hat{n}-\hat{z}|}\mathcal{L}_{n'}^{\sigma}\bigg\{T_{\hat{x}}\widetilde{S}^{\sigma}_{1}(\hat{n},\hat{z})T_{\hat{x}}\widetilde{G}_\sigma(\hat{z})\bigg\}\,\mathrm{d}s(\hat{z})+\int_{\mathbb{S}^2}\mathcal{L}_{n'}^{\sigma}\bigg\{T_{\hat{x}}\widetilde{S}^{\sigma}_{2}(\hat{n},\hat{z})T_{\hat{x}}\widetilde{G}_\sigma(\hat{z})\bigg\}\,\mathrm{d}s(\hat{z})\\
=&\sum_{r'=0}^{2n'+1}\sum_{s'=1}^{n'}\xi_{r'}\eta_{s'}\Big[\alpha_{s'}^{n'}T_{\hat{x}}\widetilde{S}^{\sigma}_{1}(\hat{n},\hat{y}_{r's'})+T_{\hat{x}}\widetilde{S}^{\sigma}_{2}(\hat{n},\hat{y}_{r's'})\Big]
T_{\hat{x}}\widetilde{G}_\sigma(\hat{y}_{r's'}).
\end{align*} 
Here $\alpha_{s'}^{n'}:=\sum_{l=0}^{n'}P_l(\hat{n}\cdot\hat{y}_{r's'})$, $\xi_{r'}=\mu_{r'}$ and $\eta_{s'}=\nu_{s'}$.

In view of Theorems \ref{ThN} and \ref{ThH}, the Galerkin method for \eqref{boundaryoperator} seeks to approximate solutions $\widetilde{G}^{n}_{\mathfrak{p}}\in X_n$ and $\widetilde{G}^{n}_{\mathfrak{s}}\in \underline{\mathbb{T}_n}$, which can be written as
\begin{align*}
	\widetilde{G}^{n}_{\mathfrak{p}}(\hat{x})=\sum_{l=0}^n\sum_{j=-l}^{l}w_{lj}Y_{l,j}(\hat{x}), \quad \widetilde{G}^{n}_{\mathfrak{s}}(\hat{x})=\sum_{l=0}^n\sum_{j=-l}^{l}\sum_{\tilde{k}=1}^{2}W_{lj\tilde{k}}\boldsymbol{Z}_{l,j}^{(\tilde{k})}(\hat{x}),
\end{align*}
and satisfy 
\begin{align}\label{boundaryIE2_Galerkin}
	\begin{cases}
&-(J_{\boldsymbol q}\widetilde{G}^{n}_{\mathfrak{p}},Y_{l',j'})_{n+1}+(\mathcal{K}_{n'}\widetilde{G}^{n}_{\mathfrak{p}},Y_{l',j'})_{n+1}\\
&\hspace{2cm}+\big(\mathcal{S}_{n'}^{\mathfrak{s}}\widetilde{G}^{n}_{\mathfrak{s}},[{\bf D}_{\mathbb S^2}\boldsymbol{q}]{\bf curl}_{\mathbb{S}^2}Y_{l',j'}\big)_{n+1}	=	2(\tilde{f}_1 J_{\boldsymbol q},Y_{l',j'})_{n+1},\\ 
&\big(\mathcal{S}^{\mathfrak{p}}_{n'}\widetilde{G}^{n}_{\mathfrak{p}},[{\bf D}_{\mathbb{S}^2}\bq]^\top:{\bf curl}_{\mathbb{S}^2}\boldsymbol{Z}_{l',j'}^{(1)}\big)_{n+1}+(J_{\boldsymbol q}\widetilde{G}^{n}_{\mathfrak{s}},\boldsymbol{Z}_{l',j'}^{(1)})_{n+1}\\
&\hspace{2cm}+(\mathcal{M}_{n'}\widetilde{G}^{n}_{\mathfrak{s}},\boldsymbol{Z}_{l',j'}^{(1)})_{n+1} = 2(\boldsymbol{\tilde{f}}_2J_{\boldsymbol q},\boldsymbol{Z}_{l',j'}^{(1)})_{n+1},
\\
&\big(\mathcal{S}^{\mathfrak{p}}_{n'}\widetilde{G}^{n}_{\mathfrak{p}},[{\bf D}_{\mathbb{S}^2}{\boldsymbol q}]^\top:{\bf curl}_{\mathbb{S}^2}\boldsymbol{Z}_{l',j'}^{(2)}\big)_{n+1}+(J_{\boldsymbol q}\widetilde{G}^{n}_{\mathfrak{s}},\boldsymbol{Z}_{l',j'}^{(2)})_{n+1}\\
&\hspace{2cm}+(\mathcal{M}_{n'}\widetilde{G}^{n}_{\mathfrak{s}},\boldsymbol{Z}_{l',j'}^{(2)})_{n+1} = 2(\boldsymbol{\tilde{f}}_2J_{\boldsymbol q},\boldsymbol{Z}_{l',j'}^{(2)})_{n+1}
\end{cases}
\end{align}
for $l'=0,1,\cdots,n, ~|j'|\leq l'$, where   $\tilde{f}_1=f_1\circ\boldsymbol{q}$, $\boldsymbol{\tilde{f}}_2=\boldsymbol{f}_2\circ\boldsymbol{q}$ and $n' = an+1$ with $a>1$.

To assemble the matrix, we denote the corresponding matrix elements in \eqref{boundaryIE2_Galerkin} by 
\begin{align*}
\mathbf{N}^{\tilde{k}}_{l'j',lj}&:=\big(\mathcal{S}_{n'}^{\mathfrak{s}}\boldsymbol{Z}^{(\tilde{k})}_{l,j},[{\bf D}_{\mathbb S^2}\boldsymbol{q}]{\bf curl}_{\mathbb{S}^2}Y_{l',j'}\big)_{n+1},\quad \tilde{k}=1,2, \\
\mathbf{H}^{k'}_{l'j',lj}&:=\big(\mathcal{S}^{\mathfrak{p}}_{n'}Y_{l,j},[{\bf D}_{\mathbb{S}^2}\bq]^\top:{\bf curl}_{\mathbb{S}^2}\boldsymbol{Z}_{l',j'}^{(k')}\big)_{n+1}, \quad k'=1,2, \\
\mathbf{U}^{k',\tilde{k}}_{l'j',lj}&:=\big(J_{\boldsymbol q}\boldsymbol{Z}^{(\tilde{k})}_{l,j},\boldsymbol{Z}_{l',j'}^{(k')}\big)_{n+1},\quad
\mathbf{I}_{l'j',lj}:=\big(J_{\boldsymbol q}Y_{l,j},Y_{l',j'}\big)_{n+1}, \quad k',\tilde{k}=1,2, \\
\mathbf{M}_{l'j',lj}^{k',\tilde{k}}&:=(\mathcal{M}_{n'}\boldsymbol{Z}_{lj}^{\tilde{(k)}},\boldsymbol{Z}_{l'j'}^{(k')})_{n+1}, \quad \mathbf{K}_{l'j',lj}:=(\mathcal{K}_{n'}Y_{l,j},Y_{l',j'})_{n+1}, ~\quad k',\tilde{k}=1,2.
\end{align*}
Let us also introduce the following notations
\begin{align}\label{points}
\hat{x}_{rs}=p(\theta_s,\varphi_r), \quad \hat{y}_{r's'}=p(\varTheta_{s'},\varPhi_{r'}), \quad \hat{y}_{rs}^{r's'}=T^{-1}_{p(\theta_s,\varphi_r)}p(\varTheta_{s'},\varPhi_{r'}):=p(\varLambda_{rs}^{r's'},\varXi_{rs}^{r's'})
\end{align}
with the quadrature knots $\theta_s$, $\varphi_r$, $\varTheta_{s'}$, $\varPhi_{r'}$. Note that $\sin\theta_s\not=0$ and $\sin\varTheta_{s'}\not=0$ as taking Gaussian quadrature nodes. Since the singularity is transferred to the north pole, we need the representation of rotated spherical harmonics. It follows from standard calculations\cite{GH2007} that 
\begin{align}
Y_{l,j}(\hat{y}_{rs}^{r's'})&=\sum_{\tilde{j}=- l}^{l}F_{sl\tilde{j}j}\mathrm{e}^{\mathrm{i}(j-\tilde{j})\varphi_r}Y_{l,\tilde{j}}(\hat{y}_{r's'}),\label{Y_lj1}\\
\boldsymbol{Z}^{(\tilde{k})}_{l,j}(\hat{y}_{rs}^{r's'})&=\mathcal{F}(\hat{y}_{rs}^{r's'})\sum_{\tilde{j}=- l}^{l}F_{sl\tilde{j}j}\mathrm{e}^{\mathrm{i}(j-\tilde{j})\varphi_r}\sum_{d=1}^{2}\alpha_{l,j}^{(\tilde{k},d)}(\varTheta_{s'})\mathrm{e}^{\mathrm{i}\tilde{j}\varPhi_{r'}}T^{-1}_{\hat{x}_{rs}}\boldsymbol{v}^{(d)}(\varTheta_{s'},\varPhi_{r'}),\label{Z_lj1}
\end{align}
where
\begin{align*}
	F_{sl\tilde{j}j}&=\mathrm{e}^{\mathrm{i}(j-\tilde{j})\pi/2}\sum_{m=-l}^{l}d_{\tilde{j}m}^{(l)}(\pi/2)d_{jm}^{(l)}(\pi/2)\mathrm{e}^{\mathrm{i}m\theta_s}, \\
	\boldsymbol{v}^{(1)}(\theta,\phi) &= (\cos \theta\cos\phi,\cos\theta\sin\phi,-\sin\theta)^T, \\
	\boldsymbol{v}^{(2)}(\theta,\phi) &= (-\cos\phi,\cos\theta,0)^T,\\
	\alpha_{l,j}^{(1,1)}(\theta) &= \alpha_{l,j}^{(2,2)}(\theta) = \frac{1}{\sqrt{l(l+1)}}c_l^j\frac{\partial P_l^{|j|}(\cos\theta)}{\partial \theta},\\
	\alpha_{l,j}^{(1,2)}(\theta) &= \alpha_{l,j}^{(2,1)}(\theta) = \frac{1}{\sqrt{l(l+1)}}c_l^j\frac{\mathrm{i}j}{\sin \theta}  P_l^{|j|}(\cos\theta). 
\end{align*}
Here
$$
d_{jm}^{(l)}(\pi/2)=2^j\sqrt{\frac{(l+j)!(l-j)!}{(l+m)!(l-m)!}}\mathcal{P}_{l+j}^{(m-j,-m-j)}(0),
$$
and $\mathcal{P}_{\bar{n}}^{(\alpha,\beta)}$ is the normalized Jacobi polynomial given by 
$$
\mathcal{P}_{\bar{n}}^{(\alpha,\beta)}(0)=2^{-\bar{n}}\sum_{\bar{t}=0}^{\bar{n}}(-1)^{\bar{t}}\left(
\begin{array}{c}
\bar{n}+\alpha\\
\bar{n}-\bar{t}
\end{array}\right)
\left(
\begin{array}{c}
\bar{n}+\beta\\
\bar{t}
\end{array}\right), \quad \alpha\geq0, \quad \beta\geq0.
$$
If $m-j$ or $-m-j$ is negative, then the following symmetry relation can be used to compute $d_{jm}^{(l)}(\pi/2)$:  
$$
d_{jm}^{(l)}(\varphi)=(-1)^{j-m}d_{mj}^{(l)}(\varphi)=d_{-m-j}^{(l)}(\varphi)=d_{mj}^{(l)}(-\varphi).
$$

Combining \eqref{points}--\eqref{Z_lj1} and \eqref{curls2},  we find that the element $\mathbf{N}^{\tilde{k}}_{l'j',lj}$ can be evaluated by
\begin{align*}
	\mathbf{N}^{\tilde{k}}_{l'j',lj}&=\sum_{r=0}^{2n+3}\sum_{s=1}^{n+2}\mu_r\nu_s\sum_{r'=0}^{2n'+1}\sum_{s'=1}^{n'+1}\xi_{r'}\eta_{s'}\Big[\alpha_{s'}^{n'}\widetilde{S}^{\mathfrak{s}}_{1}(\hat{x}_{rs},\hat{y}_{rs}^{r's'})+\widetilde{S}^{\mathfrak{s}}_{2}(\hat{x}_{rs},\hat{y}_{rs}^{r's'})\Big]\\
	&\quad\sum_{\tilde{j}=- l}^{l}F_{sl\tilde{j}j}\mathrm{e}^{\mathrm{i}(j-\tilde{j})\varphi_r}\sum_{d=1}^{2}\alpha_{l,j}^{(\tilde{k},d)}(\varTheta_{s'})\mathrm{e}^{\mathrm{i}\tilde{j}\varPhi_{r'}}\Big(\mathcal{F}(\hat{y}_{rs}^{r's'})T^{-1}_{\hat{x}_{rs}}\boldsymbol{v}^{(d)}(\varTheta_{s'},\varPhi_{r'})\Big) \\
	&\quad\cdot\Big(-c_{l'}^{j'}\frac{\partial P_{l'}^{|j'|}(\cos\theta_s)}{\partial\theta_s}\me^{-\mi j'\varphi_r}\boldsymbol{t}_2(\hat{x}_{rs})-\mi j'c_{l'}^{j'}P_{l'}^{|j'|}(\cos\theta_s)\me^{-\mi j'\varphi_r}\frac{1}{\sin\theta_s}\boldsymbol{t}_1(\hat{x}_{rs})\Big)\\
\end{align*}
for $k'=1,2$, and the element $\mathbf{H}^{\tilde{k}}_{l'j',lj}$ is given by
\begin{align*}
\mathbf{H}^{k'}_{l'j',lj}&=\sum_{r=0}^{2n+3}\sum_{s=1}^{n+2}\mu_r\nu_s\sum_{r'=0}^{2n'+1}\sum_{s'=1}^{n'+1}\xi_{r'}\eta_{s'}\Big[\alpha_{s'}^{n'}\widetilde{S}^{\mathfrak{p}}_{1}(\hat{x}_{rs},\hat{y}_{rs}^{r's'})+\widetilde{S}^{\mathfrak{p}}_{2}(\hat{x}_{rs},\hat{y}_{rs}^{r's'})\Big]\\
&\quad\sum_{\tilde{j}=- l}^{l}F_{sl\tilde{j}j}\mathrm{e}^{\mathrm{i}(j-\tilde{j})\varphi_r}c_l^{\tilde{j}}P_l^{|\tilde{j}|}(\cos\varTheta_{s'})\mathrm{e}^{\mathrm{i}j\varPhi_{r'}}
\\
&\quad\Big(\boldsymbol{t}_1(\hat{x}_{rs})\cdot\Big[\frac{1}{\sin\theta_{s}}\frac{\partial\overline{Z}^{(k')}_{l',j'}\circ p(\theta_{s},\varphi_{r})}{\partial\varphi_r}\Big]-\boldsymbol{t}_2(\hat{x}_{rs})\cdot\frac{\partial\overline{Z}^{(k')}_{l',j'}\circ p(\theta_{s},\varphi_{r})}{\partial\theta_{s}}\Big)
\end{align*}
for $\tilde{k}=1,2$. It can be seen that the direct computation for each element needs $\mathcal{O}(n^4)$ computational cost, which leads to $\mathcal{O}(n^8)$ total computational complexity since there are $\mathcal{O}(n^4)$ matrix elements. To accelerate the evaluation, we take the idea of \cite{GH2008, GS2002} and make the following decomposition:
\begin{align*}
&E^{1,d}_{srs'\tilde{j}}=-\sum_{r'=0}^{2n'+1}\xi_{r'}\mathrm{e}^{\mathrm{i}\tilde{j}\varPhi_{r'}}\Big[\alpha_{s'}^{n'}\widetilde{S}^{\mathfrak{p}}_{1}(\hat{x}_{rs},\hat{y}_{rs}^{r's'})+\widetilde{S}^{\mathfrak{p}}_{2}(\hat{x}_{rs},\hat{y}_{rs}^{r's'})\Big]\Big(\mathcal{F}(\hat{y}_{rs}^{r's'})T^{-1}_{\hat{x}_{rs}}\boldsymbol{v}^{(d)}(\varTheta_{s'},\varPhi_{r'})\Big)\cdot\boldsymbol{t}_2(\hat{x}_{rs}), \\
&E^{2,d}_{srs'\tilde{j}}=-\sum_{r'=0}^{2n'+1}\xi_{r'}\mathrm{e}^{\mathrm{i}\tilde{j}\varPhi_{r'}}\Big[\alpha_{s'}^{n'}\widetilde{S}^{\mathfrak{p}}_{1}(\hat{x}_{rs},\hat{y}_{rs}^{r's'})+\widetilde{S}^{\mathfrak{p}}_{2}(\hat{x}_{rs},\hat{y}_{rs}^{r's'})\Big]\Big(\mathcal{F}(\hat{y}_{rs}^{r's'})T^{-1}_{\hat{x}_{rs}}\boldsymbol{v}^{(d)}(\varTheta_{s'},\varPhi_{r'})\Big)\cdot\boldsymbol{t}_1(\hat{x}_{rs}), \\
&D^{\tilde{k},h}_{srl\tilde{j}}=\sum_{s'=1}^{n'+1}\sum_{d=1}^{2}\eta_{s'}\alpha^{(\tilde{k},d)}_{l,\tilde{j}}(\varTheta_{s'})E^{h,d}_{srs'\tilde{j}},\quad h=1,2,\\
&C^{\tilde{k},h}_{srlj}=\sum_{|\tilde{j}|\leq l} F_{sl\tilde{j}j}\mathrm{e}^{\mathrm{i}(j-\tilde{j})\varphi_r}D^{\tilde{k},h}_{srl\tilde{j}},\quad
B^{\tilde{k},h}_{sj'lj}=\sum_{r=0}^{2n+3}\mu_r\mathrm{e}^{-\mathrm{i}j'\varphi_r}C^{\tilde{k},h}_{srlj},\quad h=1,2,
\end{align*}
which lead to 
\begin{align*}
\mathbf{N}^{\tilde{k}}_{l'j',lj}=\sum_{s=1}^{n+2}\nu_s\Big(c_{l'}^{j'}\frac{\partial P_{l'}^{|j'|}(\cos\theta_s)}{\partial\theta_s}B^{\tilde{k},1}_{sj'lj}+\mi j'c_{l'}^{j'}P_{l'}^{|j'|}(\cos\theta_{s})\frac{1}{\sin\theta_s}B^{\tilde{k},2}_{sj'lj}\Big).
\end{align*}
Since each step in the decomposition only requires $\mathcal{O}(n)$ amount of work, there is $\mathcal{O}(n^5)$ computational cost for $\mathcal{O}(n^4)$ matrix elements.  It is a huge cost saving compared to the evaluation without decomposition. 

Similarly, we may make the following decomposition for $\mathbf{H}^{k'}_{l'j',lj}$: 
\begin{align*}
&E^{1,d}_{srs'\tilde{j}}=\sum_{r'=0}^{2n'+1}\xi_{r'}\mathrm{e}^{\mathrm{i}\tilde{j}\varPhi_{r'}}\Big[\alpha_{s'}^{n'}\widetilde{S}^{\mathfrak{p}}_{1}(\hat{x}_{rs},\hat{y}_{rs}^{r's'})+\widetilde{S}^{\mathfrak{p}}_{2}(\hat{x}_{rs},\hat{y}_{rs}^{r's'})\Big]\Big(\mathcal{F}(\hat{x}_{rs})\boldsymbol{v}^{(d)}(\theta_{s},\varphi_{r})\Big)\cdot\boldsymbol{t}_1(\hat{x}_{rs})/\sin\theta_s, \\
&E^{2,d}_{srs'\tilde{j}}=\sum_{r'=0}^{2n'+1}\xi_{r'}\mathrm{e}^{\mathrm{i}\tilde{j}\varPhi_{r'}}\Big[\alpha_{s'}^{n'}\widetilde{S}^{\mathfrak{p}}_{1}(\hat{x}_{rs},\hat{y}_{rs}^{r's'})+\widetilde{S}^{\mathfrak{p}}_{2}(\hat{x}_{rs},\hat{y}_{rs}^{r's'})\Big]\frac{\partial\Big(\mathcal{F}(\hat{x}_{rs})\boldsymbol{v}^{(d)}(\theta_{s},\varphi_{r})\Big)}{\partial\varphi_r}\cdot\boldsymbol{t}_1(\hat{x}_{rs})/\sin\theta_s, \\
&E^{3,d}_{srs'\tilde{j}}=-\sum_{r'=0}^{2n'+1}\xi_{r'}\mathrm{e}^{\mathrm{i}\tilde{j}\varPhi_{r'}}\Big[\alpha_{s'}^{n'}\widetilde{S}^{\mathfrak{p}}_{1}(\hat{x}_{rs},\hat{y}_{rs}^{r's'})+\widetilde{S}^{\mathfrak{p}}_{2}(\hat{x}_{rs},\hat{y}_{rs}^{r's'})\Big]\Big(\mathcal{F}(\hat{x}_{rs})\boldsymbol{v}^{(d)}(\theta_{s},\varphi_{r})\Big)\cdot\boldsymbol{t}_2(\hat{x}_{rs}), \\
&E^{4,d}_{srs'\tilde{j}}=-\sum_{r'=0}^{2n'+1}\xi_{r'}\mathrm{e}^{\mathrm{i}\tilde{j}\varPhi_{r'}}\Big[\alpha_{s'}^{n'}\widetilde{S}^{\mathfrak{p}}_{1}(\hat{x}_{rs},\hat{y}_{rs}^{r's'})+\widetilde{S}^{\mathfrak{p}}_{2}(\hat{x}_{rs},\hat{y}_{rs}^{r's'})\Big]\frac{\partial\Big(\mathcal{F}(\hat{x}_{rs})\boldsymbol{v}^{(d)}(\theta_{s},\varphi_{r})\Big)}{\partial\theta_s}\cdot\boldsymbol{t}_2(\hat{x}_{rs}), \\
&D^{w,d}_{srl\tilde{j}}=\sum_{s'=1}^{n'+1}\eta_{s'}c_l^{\tilde{j}} P_l^{|\tilde{j}|}(\cos\varTheta_{s'})E^{w,d}_{srs'\tilde{j}},\quad w=1,2,3,4,\\
&C^{w,d}_{srlj}=\sum_{|\tilde{j}|\leq l} F_{sl\tilde{j}j}\mathrm{e}^{\mathrm{i}(j-\tilde{j})\varphi_r}D^{w,d}_{srl\tilde{j}},\quad
B^{w,d}_{sj'lj}=\sum_{r=0}^{2n+3}\mu_r\mathrm{e}^{-\mathrm{i}j'\varphi_r}C^{w,d}_{srlj},\quad w=1,2,3,4.
\end{align*}
Here we have $\mathcal{F}^\top(\hat{x}_{rs})=[F_1(\hat{x}_{rs}),F_2(\hat{x}_{rs}),F_3(\hat{x}_{rs})]$, and then 
$$
\frac{\partial}{\partial\theta_{s}}\Big(\mathcal{F}(\hat{x}_{rs})\boldsymbol{v}^{(d)}(\theta_{s},\varphi_{r})\Big)=\left[ \begin{array}{c}
	\partial_{\theta_{s}}F_1(\hat{x}_{rs})\cdot \boldsymbol{v}^{(d)}(\theta_{s},\varphi_{r})+F_1(\hat{x}_{rs})\cdot\partial_{\theta_{s}}\boldsymbol{v}^{(d)}(\theta_{s},\varphi_{r}) \\
	\partial_{\theta_{s}}F_2(\hat{x}_{rs})\cdot \boldsymbol{v}^{(d)}(\theta_{s},\varphi_{r})+F_2(\hat{x}_{rs})\cdot\partial_{\theta_{s}}\boldsymbol{v}^{(d)}(\theta_{s},\varphi_{r})\\
	\partial_{\theta_{s}}F_3(\hat{x}_{rs})\cdot \boldsymbol{v}^{(d)}(\theta_{s},\varphi_{r})+F_3(\hat{x}_{rs})\cdot\partial_{\theta_{s}}\boldsymbol{v}^{(d)}(\theta_{s},\varphi_{r})
\end{array} 
\right ].
$$
Similarly, one can obtain $\frac{\partial}{\partial\varphi_r}\Big(\mathcal{F}(\hat{x}_{rs})\boldsymbol{v}^{(d)}(\theta_{s},\varphi_{r})\Big)$.
Then $\mathbf{H}^{k'}_{l'j',lj}$ can be evaluated by
\begin{align*}
\mathbf{H}^{k'}_{l'j',lj}=\sum_{s=1}^{n+2}\sum_{d=1}^2\nu_s\Big(\overline{\mi j'\alpha^{(k',d)}_{l',j'}(\theta_s)}B^{1,d}_{sj'lj}+\overline{\alpha^{(k',d)}_{l',j'}(\theta_s)}B^{2,d}_{sj'lj}+\overline{\frac{\partial\alpha^{(k',d)}_{l',j'}(\theta_s)}{\partial\theta_s}}B^{3,d}_{sj'lj}+\overline{\alpha^{(k',d)}_{l',j'}(\theta_s)}B^{4,d}_{sj'lj}\Big).
\end{align*}
It is worth pointing out that the numerical implementation can be done very efficiently since each step of the operations only involves scalar functions. 

We briefly mention the evaluation of $\mathbf{U}^{k',\tilde{k}}_{l'j',lj}$ and $\mathbf{I}_{l'j',lj}$. Noting that $\mathcal{F}$ is a orthogonal transformation, we may obtain $\mathbf{U}^{k',\tilde{k}}_{l'j',lj}$ via
\begin{align*}
&C^{\tilde{k},d'}_{srlj}=\alpha_{l,j}^{(\tilde{k},d')}(\theta_s)\me^{\mi j \varphi_{r}}J_{\boldsymbol q}\circ p(\theta_s,\varphi_r),\quad
B^{\tilde{k},d'}_{sj'lj}=\sum_{r=0}^{2n+3}\mu_r\mathrm{e}^{-\mathrm{i}j'\varphi_r}C^{\tilde{k},d'}_{srlj},\quad\tilde{k},d'=1,2,\\
&\mathbf{U}^{k',\tilde{k}}_{l'j',lj}=\sum_{s=1}^{n+2}\sum_{d'=1}^{2}\nu_s\overline{\alpha^{(k',d')}_{l',j'}(\theta_{s})}B^{\tilde{k},d'}_{sj'lj},
\end{align*}
and  $\mathbf{I}_{l'j',lj}$ via 
\begin{align*}
&C_{srlj}=c_l^jP_l^{|j|}(\cos\theta_s)\me^{\mi j \varphi_{r}}J_{\boldsymbol q}\circ p(\theta_s,\varphi_r),\quad
B_{sj'lj}=\sum_{r=0}^{2n+3}\mu_r\mathrm{e}^{-\mathrm{i}j'\varphi_r}C_{srlj},\\
&\mathbf{I}_{l'j',lj}=\sum_{s=1}^{n+2}\nu_sc_{l'}^{j'}P_{l'}^{|j'|}(\cos\theta_s)B_{sj'lj}.
\end{align*}
The approximations $\mathbf{K}_{l'j',lj}$ and $\mathbf{M}^{k',\tilde{k}}_{l'j',lj}$ for $\mathcal{K}$ and $\mathcal{M}$ are discussed in \cite{GH2008, GS2002}. We give them in the appendix for completeness. 

\begin{remark}
The convergence analysis of the proposed numerical method depends on the invertibility of the boundary integral system \eqref{boundaryIE0} as well as the discretized system \eqref{boundaryIE2_Galerkin}, which is beyond the scope of this paper and currently under investigation. We refer to \cite{DLL2021} for the convergence analysis for the two-dimensional problems. It is expected that the following estimate holds under certain conditions: 
\begin{eqnarray}
	||(\widetilde{G}^{n}_{\mathfrak{p}}\circ \boldsymbol{q}^{-1} ,\widetilde{G}^{n}_{\mathfrak{s}}\circ \boldsymbol{q}^{-1})^\top-(g_1,\bg_2)^\top||_{\infty,\Gamma_{D}}\le \frac{C}{n^{q-1}} ||(g_1,\bg_2)^\top||_{q,\infty,\Gamma_{D}},\quad \forall q\in\mathbb{N},
\end{eqnarray}
where $(g_1,\bg_2)^\top$ is the exact solution to the integral equation system \eqref{boundaryoperator}, $||\cdot||_{\infty,\Gamma_{D}}$ is the maximum norm for functions in $C(\Gamma_{D})\otimes {\boldsymbol C}_T(\Gamma_{D})$, and $||\cdot||_{q,\infty,\Gamma_{D}}$ is the norm for continuously differentiable vector functions on $\Gamma_{D}$ up to order $q$. 
Numerical experiments show that the spectral convergence is achieved as long as $n'\ge 2n+1$. 
\end{remark}

\section{Numerical experiments}

In this section, we present some numerical experiments to demonstrate the superior performance of the proposed method. We consider three different geometries for the obstacle: ellipsoid, cushion, and bean, as shown in  Fig. \ref{obastacles}. Their parametrizations are given in Table \ref{boundary}. Throughout the numerical experiments, we take the Lam\'{e} parameters $\lambda=2, \mu=1$ and the truncation number $n'=2n+1$. The method is implemented using MATLAB on a server with two Intel Xeon cores and 256 GB RAM. No special effort is paid to solving the resulted linear system of equations other than the $backslash$  command in MATLAB.

\begin{figure}[h]
	\centering 
	\subfigure[Ellipsoid]
	{\includegraphics[width=0.32\textwidth]{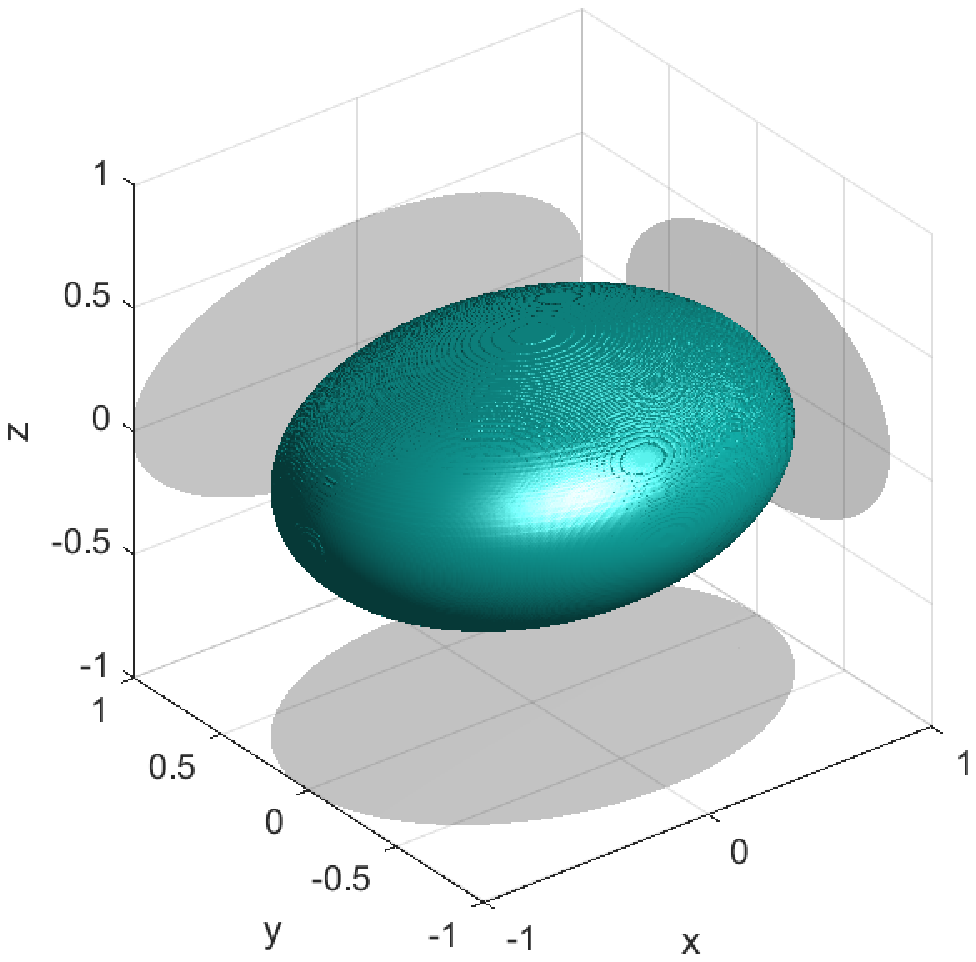}} 
	\subfigure[Cushion]
	{\includegraphics[width=0.32\textwidth]{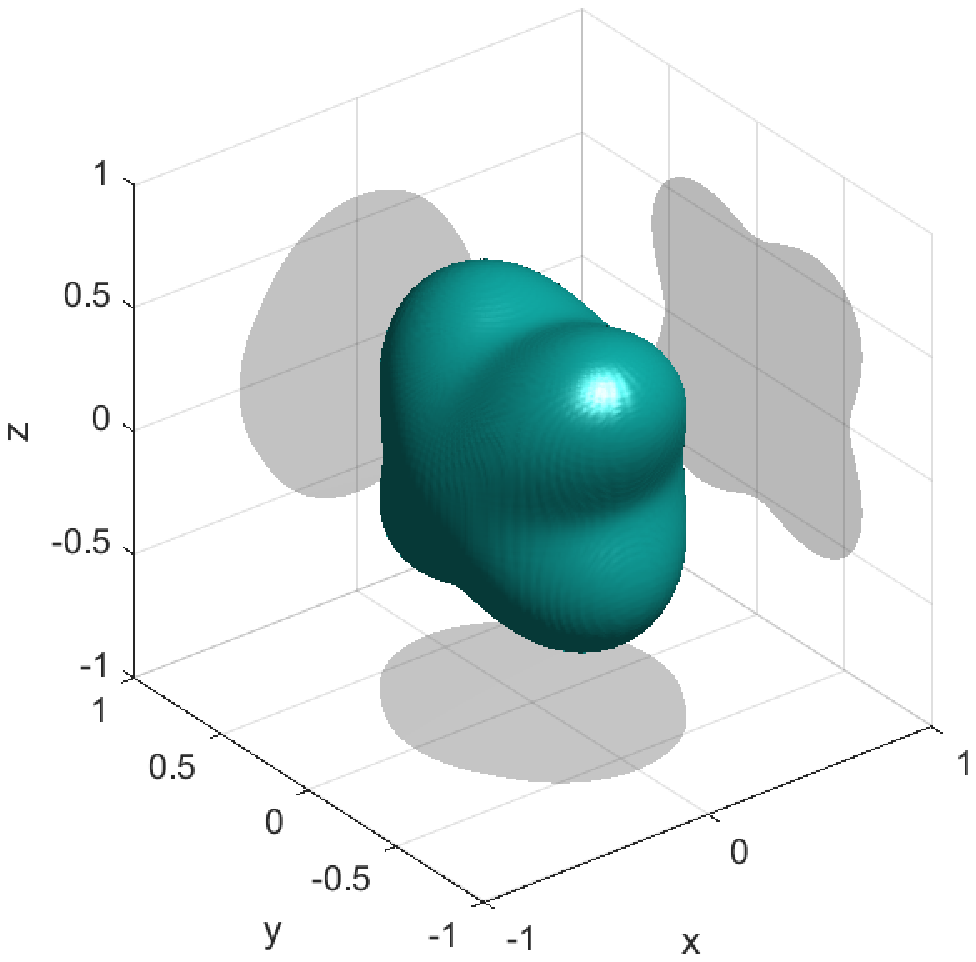}}
	\subfigure[Bean]
	{\includegraphics[width=0.32\textwidth]{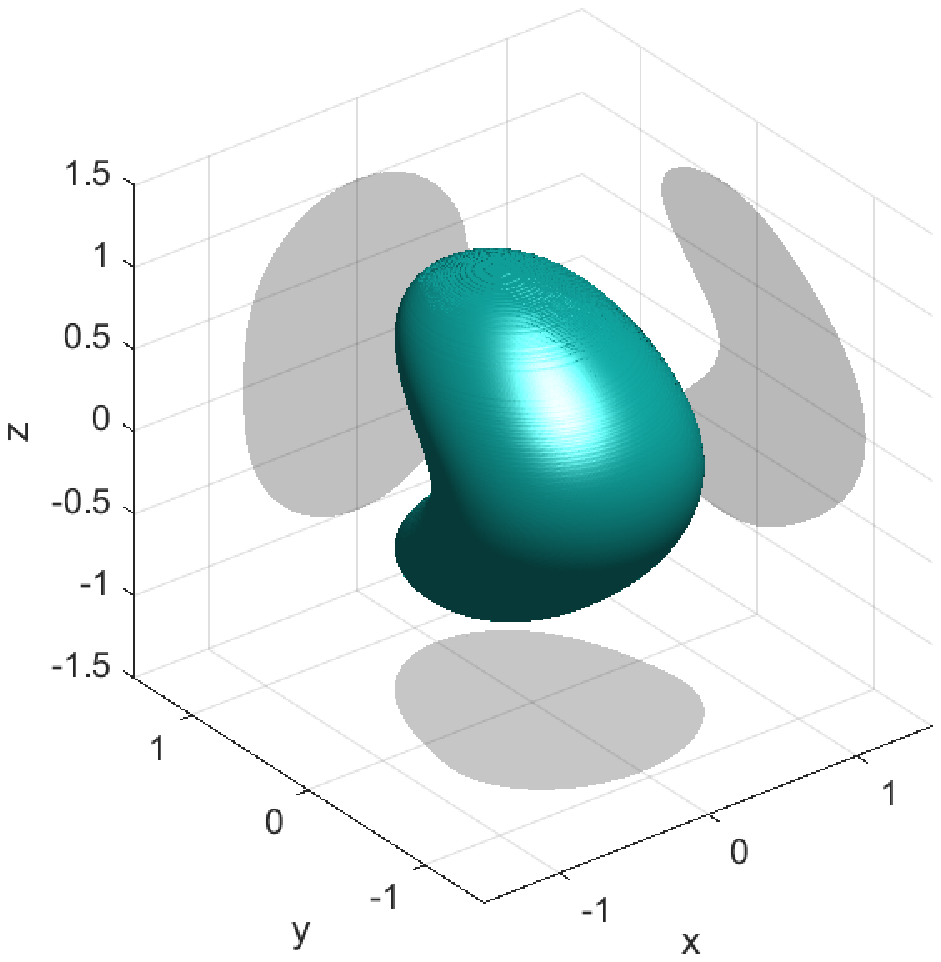}}
	\caption{Geometries of the obstacles.}\label{obastacles}
\end{figure}

\begin{table}[h] 
	\caption{Parametrizations of the obstacles.}
	\label{boundary}
	\begin{tabular}{lll}
		\toprule[1pt]
		Type           & Parametrization\\
		\midrule \vspace{1.5ex} 
		Ellipsoid  & 
		$\displaystyle
		x_1^2+\frac{x_2^2}{0.75^2}+\frac{x_3^2}{0.5^2}=1$
		\vspace{1.5ex}
		\\ 
		Cushion  & 
		$\displaystyle z(\theta,\varphi)=\sqrt{0.27+0.065(\cos2\varphi-1)(\cos4\theta-1)}\hat{x}(\theta,\varphi)$
		\vspace{1.5ex}\\
		Bean  & 
		$\displaystyle
		\frac{x_1^2}{0.64\big(1-0.1\cos(\pi x_3)\big)}+\frac{\big(0.3\cos(\pi x_3)+x_2\big)^2}{0.64\big(1-0.4\cos(\pi x_3)\big)}+x_3^2=1$ \\
		\bottomrule[1pt]
	\end{tabular}
\end{table}

To test the accuracy of the proposed method, we construct an exact solution in form of \begin{align}\label{exact solution1}
	\boldsymbol{v_*}(x)=\mathbb{G}(x,y_0)\boldsymbol{p},\quad  y_0=(0,0.05,0.0866)^\top, \quad \boldsymbol{p}=(1,0,0)^\top,
\end{align}
which is the same as the first test in \cite{Louer2014}, where the tensor 
$$
\mathbb{G}(x,y)=\frac{1}{\mu}\Big(\Phi(x,y;\kappa_{\mathfrak s})I+\frac{1}{\kappa_{\mathfrak s}^2}\nabla_x \nabla_x^\top\big(\Phi(x,y;\kappa_{\mathfrak s})-\Phi(x,y;\kappa_{\mathfrak p})\big)\Big)
$$ 
is the fundamental solution of the elastic wave equation and $\Phi(x, y, \kappa)$ is the fundamental solution for the three-dimensional Helmholtz equation given in \eqref{fs3d}. Then, the corresponding far-field is given by
$$
\boldsymbol{v_{*,ps}^\infty}(\hat{x})=\frac{1}{\mu}\frac{\mathrm{e}^{-\mathrm{i}\kappa_{\mathfrak s}\hat{x}\cdot y_0}}{4\pi}(\hat{x}\times\boldsymbol{p})\times\hat{x}+\frac{1}{\lambda+2\mu}\frac{\mathrm{e}^{-\mathrm{i}\kappa_{\mathfrak p}\hat{x}\cdot y_0}}{4\pi}(\hat{x}\cdot\boldsymbol{p})\hat{x}.
$$
Due to the uniqueness result given in Theorem \ref{unique1}, we can solve the boundary value problem \eqref{HelmholtzDec} by enforcing the following boundary conditions on $\Gamma_D$:
\begin{align*}
	\boldsymbol{u} = \boldsymbol{v}_*. 
\end{align*}
Then, the numerical far-field pattern  $\boldsymbol{v}_n^\infty=\boldsymbol{v}_{n,\mathfrak{p}}^\infty+\boldsymbol{v}_{n,\mathfrak{s}}^\infty$ can be calculated by using \eqref{behaviour relation} and \eqref{singlelayer_far}.
The maximum errors are calculated over 1300 observations (equally spaced for the observation angles $\theta$ and $\varphi$) in accordance with the expression
$$
\|\epsilon_{ps}\|_\infty:=\|\boldsymbol{v}_{n,ps}^\infty-\boldsymbol{v}_{*,ps}^\infty\|_\infty=\max_{\hat{x}\in\mathbb{S}^2}|\boldsymbol{v}_{n,ps}^\infty(\hat{x})-\boldsymbol{v}_{*,ps}^\infty(\hat{x})|.
$$

In addition to the point source case,  we also compute the far-field pattern, denoted by $\boldsymbol{v}^\infty_{pw}$, resulted from the elastic plane wave incidence
$$
\boldsymbol{u}^i(x)=\frac{1}{\mu}\mathrm{e}^{\mathrm{i}\kappa_{\mathfrak s}x\cdot \boldsymbol{d}}(\boldsymbol{d}\times\boldsymbol{p})\times\boldsymbol{d}+\frac{1}{\lambda+2\mu}\mathrm{e}^{\mathrm{i}\kappa_{\mathfrak p}x\cdot \boldsymbol{d}}(\boldsymbol{d}\cdot\boldsymbol{p})\boldsymbol{d}, \quad \boldsymbol{d},\boldsymbol{p}\in\mathbb{S}^2,
$$ 
where the incident direction vector $\boldsymbol{d}=(0,0,1)^\top$ and the polarization vector $\boldsymbol{p}=(1,0,0)^\top$. Again, we calculate the maximum errors over the observations on the unit sphere by using
$$
\|\epsilon_{pw}\|_\infty:=\|\boldsymbol{v}_{n,pw}^\infty-\boldsymbol{v}_{n_*,pw}^\infty\|_\infty=\max_{\hat{x}\in\mathbb{S}^2}|\boldsymbol{v}_{n,pw}^\infty(\hat{x})-\boldsymbol{v}_{n_*,pw}^\infty(\hat{x})|.
$$
where $n_*$ is a sufficiently large number.

\subsection{Example 1}

In this example, we evaluate the elastic scattering problem for three different obstacles at the frequency $\omega = \pi$.  
We choose $n_*=60$ when the analytical solution is not available. Numerical results for the ellipsoid-shaped obstacle are given in Table \ref{error_ellipsoid}. It is shown that the solver rapidly achieves 13 digits accuracy for the point source test with $n=35$ and stops increasing due to the round off errors. For the plane wave scattering, $10$ digits accuracy is obtained with $n=25$. Table \ref{error_ellipsoid} also shows the time to construct the scattering matrix, denoted by $T_{coe}$ in seconds, and the time to solve the linear system, denoted by $T_{sol}$ in seconds. It is clear to note that the time is dominated by the matrix construction and roughly scales on the order of $\mathcal{O}(n^5)$, which is consistent with our complexity analysis. Tables \ref{error_cushion} and \ref{error_bean} give the numerical results for the cushion- and bean-shaped obstacles, respectively. Both tables show a rapid convergence as $n$ increases, which confirms the spectral accuracy of the solver.

\begin{table}[hhh]
	\centering 
	\caption{Numerical results for the ellipsoid-shaped obstacle at $\omega = \pi$.} 
	\label{error_ellipsoid} 
	\begin{tabular}{c|c|c|c|c}  
		\toprule[1pt]
		& \multicolumn{4}{c}{\bf Ellipsoid: \quad$\omega=\pi$}  \\ 
		\cline{2-5}
		$n$&$||\epsilon_{ps}||_\infty$
		&$||\epsilon_{pw}||_\infty$ & $T_{coe}$  & $T_{sol}$\\
		\toprule[1pt] 
		5 &\quad2.0854e-04~\quad&\quad7.8646e-03~\quad&0.2s&0.0003s \\
		15&\quad2.1597e-08~\quad&\quad6.7751e-08~\quad&2.9s&0.008s \\
		25&\quad2.6595e-12~\quad&\quad3.6123e-11~\quad&16.8s&0.1s \\
		35&\quad2.9117e-14~\quad&\quad3.3012e-11~\quad&64.3s&0.5s \\
		45&\quad5.8231e-14~\quad&\quad3.0243e-11~\quad&239.8s&2.0s \\
		55&\quad4.4362e-14~\quad&\quad4.0796e-11~\quad&874.4s&6.1s \\
		\bottomrule[1pt] 
	\end{tabular}
\end{table}

\begin{table}[hhh]
	\centering 
	\caption{Numerical results for the cushion-shaped obstacle at $\omega = \pi$.} 
	\label{error_cushion} 
	\begin{tabular}{c|c|c|c|c}  
		\toprule[1pt]
		& \multicolumn{4}{c}{\bf Cushion: \quad$\omega=\pi$}  \\ 
		\cline{2-5}
		$n$&$||\epsilon_{ps}||_\infty$
		&$||\epsilon_{pw}||_\infty$ & $T_{coe}$  & $T_{sol}$\\
		\toprule[1pt] 
		5 &\quad3.2246e-04~\quad&\quad3.9169e-02~\quad&0.3s&0.0003s \\
		15&\quad8.9083e-07~\quad&\quad5.4441e-06~\quad&3.3s&0.009s \\
		25&\quad1.5665e-09~\quad&\quad4.1725e-08~\quad&17.5s&0.1s \\
		35&\quad1.0287e-11~\quad&\quad3.6207e-10~\quad&67.1s&0.5s \\
		45&\quad2.4172e-13~\quad&\quad3.6305e-11~\quad&254.1s&2.1s \\
		55&\quad6.2876e-14~\quad&\quad4.5173e-11~\quad&930.5s&7.0s \\
		\bottomrule[1pt] 
	\end{tabular}
\end{table}
\begin{table}[hhh]
	\centering 
	\caption{Numerical results for the bean-shaped obstacle at $\omega = \pi$.} 
	\label{error_bean} 
	\begin{tabular}{c|c|c|c|c}  
		\toprule[1pt]
		& \multicolumn{4}{c}{\bf Bean: \quad$\omega=\pi$}  \\ 
		\cline{2-5}
		$n$&$||\epsilon_{ps}||_\infty$
		&$||\epsilon_{pw}||_\infty$ & $T_{coe}$  & $T_{sol}$ \\
		\toprule[1pt] 
		5 &\quad4.7523e-03~\quad&\quad1.4119e-01~\quad&0.2s&0.0003s \\
		15&\quad4.8644e-05~\quad&\quad1.7490e-04~\quad&3.2s&0.008s \\
		25&\quad2.4575e-07~\quad&\quad1.7558e-06~\quad&18.0s&0.1s \\
		35&\quad3.3519e-09~\quad&\quad2.1805e-08~\quad&66.3s&0.5s \\
		45&\quad8.9240e-11~\quad&\quad9.1735e-11~\quad&237.5s&2.2s \\
		55&\quad2.9727e-11~\quad&\quad7.0980e-11~\quad&905.0s&5.9s \\		
		\bottomrule[1pt] 
	\end{tabular}
\end{table}

\subsection{Example 2}

We consider the elastic scattering of three obstacles at higher frequency $\omega = 8\pi$. The real and the imaginary parts of the quantity $\boldsymbol{v}^\infty_{n,pw}(\boldsymbol{d})\cdot\boldsymbol{p}$, together with the errors $\|\epsilon_{ps}\|_\infty$ for three obstacles are shown in Tables \ref{pw_ellipsoid}, \ref{pw_cushion}, and \ref{pw_bean}, respectively. Similarly, we observe a rapid convergence both for the point source test and plane wave scattering when $n$ increases. For a fixed $n$, the accuracy for the scattering of the ellipsoid is higher than that of the cushion and bean. This is due to the reason that the convergence rate depends on the analyticity of the boundary for the obstacle boundary \cite{R-book2}. It is expected that a less smooth boundary may lead to slower convergence rate.

\begin{table}[hhh]
	\centering 
	\caption{Scattering by an ellipsoid-shaped obstacle at $\omega = 8\pi$.} 
	\label{pw_ellipsoid} 
	\begin{tabular}{c|c|c|c}  
		\toprule[1pt]
		& \multicolumn{3}{c}{\bf Ellipsoid: \quad$\omega=8\pi$}  \\ 
		\cline{2-4}
		$n$& $||\epsilon_{ps}||_\infty$ &
		$\Re\{ \boldsymbol{v}^\infty_{n,pw}(\boldsymbol{d})\cdot\boldsymbol{p}\}$
		&
		$\Im\{ \boldsymbol{v}^\infty_{n,pw}(\boldsymbol{d})\cdot\boldsymbol{p}\}$ \\
		\toprule[1pt]
		25&3.6217e-05 &-1.564489047510042e+00 &1.051655398026258e+01 \\
		30&1.1212e-07 &-1.564570656025764e+00 &1.051657743451597e+01 \\
		35&2.7707e-10 &-1.564570705114201e+00 &1.051657744366693e+01 \\
		40&1.9588e-12 &-1.564570705195090e+00 &1.051657744366860e+01 \\
		45&5.0535e-13 &-1.564570705193652e+00 &1.051657744366452e+01 \\
		\bottomrule[1pt] 
	\end{tabular}
\end{table}

\begin{table}[hhh]
	\centering 
	\caption{Scattering by a cushion-shaped obstacle at $\omega = 8\pi$.} 
	\label{pw_cushion} 
	\begin{tabular}{c|c|c|c}  
		\toprule[1pt]
		& \multicolumn{3}{c}{\bf Cushion: \quad$\omega=8\pi$}  \\ 
		\cline{2-4}
		$n$&$||\epsilon_{ps}||_\infty$ &
		$\Re\{ \boldsymbol{v}^\infty_{n,pw}(\boldsymbol{d})\cdot\boldsymbol{p}\}$
		&
		$\Im\{ \boldsymbol{v}^\infty_{n,pw}(\boldsymbol{d})\cdot\boldsymbol{p}\}$ \\
		\toprule[1pt]
		25&2.5022e-04 &-1.569712590870811e+00 &5.039800881189093e+00\\
		30&1.0344e-05 &-1.574459019608859e+00 &5.043376519089912e+00\\
		35&2.6371e-07 &-1.574531401615982e+00 &5.043437889542368e+00\\
		40&1.4952e-08 &-1.574531761370081e+00 &5.043437868795261e+00\\
		45&1.5917e-09 &-1.574531768527667e+00 &5.043437902490211e+00\\
		50&1.7031e-10 &-1.574531769006316e+00 &5.043437900206628e+00\\
		\bottomrule[1pt] 
	\end{tabular}
\end{table}

\begin{table}[hhh]
	\centering 
	\caption{Scattering by a bean-shaped obstacle at $\omega = 8\pi$.} 
	\label{pw_bean} 
	\begin{tabular}{c|c|c|c}  
		\toprule[1pt]
		& \multicolumn{3}{c}{\bf Bean: \quad$\omega=8\pi$}  \\ 
		\cline{2-4}
		$n$& $||\epsilon_{ps}||_\infty$ &
		$\Re\{ \boldsymbol{v}^\infty_{n,pw}(\boldsymbol{d})\cdot\boldsymbol{p}\}$
		&
		$\Im\{ \boldsymbol{v}^\infty_{n,pw}(\boldsymbol{d})\cdot\boldsymbol{p}\}$ \\
		\toprule[1pt]
		35&1.7121e-02 &-2.387421716629113e+00 &1.012728600300475e+01 \\
		40&1.5762e-03 &-2.385302268097332e+00 &1.011101559995000e+01 \\
		45&1.2892e-04 &-2.384320155063459e+00 &1.010913376752758e+01 \\
		50&3.4649e-06 &-2.384311423020166e+00 &1.010898949549377e+01 \\
		55&1.9260e-07 &-2.384312575610280e+00 &1.010899080953868e+01 \\
		\bottomrule[1pt] 
	\end{tabular}
\end{table}

\subsection{Example 3}

In this example, we consider the high frequency scattering problem, which is challenging due to the high oscillation of the solution. In particular, we apply the spectral method to test the point source scattering by the ellipsoid and cushion at $\omega=16\pi$ and $\omega=24\pi$, respectively. Numerical errors for the two obstacles at different discretization number $n$ are shown in Tables \ref{hifre_ellipsoid} and \ref{hifre_cushion}. It can be seen that the high order convergence can still be achieved at high frequencies, which demonstrates that the solver is robust for the scattering problem in both low and high frequencies.

\begin{table}[hhh]
	\centering 
	\caption{Elastic scattering for the ellipsoid-shaped obstacle at high frequencies.} 
	\label{hifre_ellipsoid} 
	\begin{tabular}{c|c|c|c}  
		\toprule[1pt]
		& \multicolumn{3}{c}{\bf Ellipsoid: \quad $||\epsilon_{ps}||_\infty$}  \\ 
		\cline{2-4}
		$\omega$ & $n=75$ & $n=80$ & $n=85$ \\
		\toprule[1pt]
		$16\pi$&\quad1.2813e-11~\quad &\quad 2.8726e-12~\quad &\quad4.4200e-11\\
		$24\pi$&\quad 3.0964e-05~\quad &\quad1.3045e-07~\quad &\quad 2.2170e-10 \\
		\bottomrule[1pt] 
	\end{tabular}
\end{table}

\begin{table}[hhh]
	\centering 
	\caption{Elastic scattering for the cushion-shaped obstacle at high frequencies.} 
	\label{hifre_cushion} 
	\begin{tabular}{c|c|c|c}  
		\toprule[1pt]
		& \multicolumn{3}{c}{\bf Cushion: \quad $||\epsilon_{ps}||_\infty$}  \\ 
		\cline{2-4}
		$\omega$ & $n=75$ & $n=80$ & $n=85$ \\
		\toprule[1pt]
		$16\pi$&\quad 9.5257e-10~\quad &\quad 4.5437e-10~\quad &\quad 4.0872e-09\\
		$24\pi$&\quad 1.5524e-05~\quad &\quad 8.9533e-07~\quad &\quad 3.0466e-08 \\
		\bottomrule[1pt] 
	\end{tabular}
\end{table}

\section{Conclusion}

In this paper, we have proposed a novel boundary integral formulation and developed a high order spectral method for solving the  elastic obstacle scattering problem in three dimensions. Based on the Helmholtz decomposition, the elastic scattering problem is reduced to a coupled boundary value problem.  The uniqueness is examined for both the coupled boundary value problem and the system of boundary integral equations. By making use of the surface differential operators and Stokes' formula, we reduce the strongly singular operators to a weakly singular operator in form of the exterior integral of the Galerkin method. In addition, all operations in the full discretization are scalar, which makes the numerical implementation much easier. Numerical experiments, including three different obstacles and high frequency scattering, are shown to demonstrate the superior performance of the proposed method. Future work includes the convergence analysis of the proposed method, the extension to other boundary conditions, and an application of the method to solve the inverse elastic scattering problems. 

\section{Appendix}

For completeness, we give the approximations $\mathbf{K}_{l'j',lj}$ and $\mathbf{M}^{k',\tilde{k}}_{l'j',lj}$ for $\mathcal{K}$ and $\mathcal{M}$. The details can be found in \cite{GH2008, GS2002}.  

The approximation $\mathcal{K}_{n'}$ to $\mathcal{K}$ can be simplified as
\begin{align*}
	(\mathcal{K}_{n'}G)(\hat{x}):=&\int_{\mathbb{S}^2}\Big(\frac{1}{|\hat{n}-\hat{z}|}\mathcal{L}_{n'}\big\{T_{\hat{x}}\widetilde{K}_{1}(\hat{n},\hat{z})T_{\hat{x}}G(\hat{z})\big\}+\mathcal{L}_{n'}\big\{T_{\hat{x}}\widetilde{K}_{2}(\hat{n},\hat{z})T_{\hat{x}}G(\hat{z})\big\}\Big)\,\mathrm{d}s(\hat{z})\nonumber\\
	=&\sum_{l=0}^{n'}\sum_{|j|\leq l}\frac{4\pi}{2l+1}\big(T_{\hat{x}}\widetilde{K}_{1}(\hat{n},\cdot)T_{\hat{x}}G(\cdot),Y_{l,j}(\cdot)\big)_{n'}Y_{l,j}(\hat{n})\nonumber\\
	&+\sum_{l=0}^{n'}\sum_{|j|\leq l}\big(T_{\hat{x}}\widetilde{K}_{2}(\hat{n},\cdot)T_{\hat{x}}G(\cdot),Y_{l,j}(\cdot)\big)_{n'}Y_{l,j}(\hat{n})\nonumber\\
	=&\sum_{r'=0}^{2n'+1}\sum_{s'=1}^{n'}\xi_{r'}\eta_{s'}\Big[\alpha_{s'}^{n'}T_{\hat{x}}\widetilde{K}_{1,d}(\hat{n},\hat{y}_{r's'})+T_{\hat{x}}\widetilde{K}_{2,d}(\hat{n},\hat{y}_{r's'})\Big]T_{\hat{x}}G(\hat{y}_{r's'}).
\end{align*}
In particular, we obtain  
\begin{align*}
	\mathbf{K}_{l'j',lj}=&(\mathcal{K}_{n'}Y_{l,j},Y_{l',j'})_{n+1}\\
	=&\sum_{r=0}^{2n+3}\sum_{s=1}^{n+2}\mu_r\nu_s\sum_{r'=0}^{2n'+1}\sum_{s'=1}^{n'+1}\xi_{r'}\eta_{s'}\Big[\alpha_{s'}^{n'}\widetilde{K}_{1}(\hat{x}_{rs},\hat{y}_{rs}^{r's'})+\widetilde{K}_{2}(\hat{x}_{rs},\hat{y}_{rs}^{r's'})\Big]\\
	&\times\sum_{|\tilde{j}|\leq l}F_{sl\tilde{j}j}\mathrm{e}^{\mathrm{i}(j-\tilde{j})\varphi_r}Y_{l,\tilde{j}}\big(p(\varTheta_{s'},\varPhi_{r'})\big)\overline{Y_{l',j'}\big(p(\theta_s,\varphi_r)\big)}
\end{align*}
via the operations
\begin{align*}
	&E^{1}_{srs'\tilde{j}}=\sum_{r'=0}^{2n'+1}\xi_{r'}\widetilde{K}_{1}(\hat{x}_{rs},\hat{y}_{rs}^{r's'})\mathrm{e}^{\mathrm{i}\tilde{j}\varPhi_{r'}},\quad E^{2}_{srs'\tilde{j}}=\sum_{r'=0}^{2n'+1}\xi_{r'}\widetilde{K}_{2}(\hat{x}_{rs},\hat{y}_{rs}^{r's'})\mathrm{e}^{\mathrm{i}\tilde{j}\varPhi_{r'}},\\
	&D_{srl\tilde{j}}=\sum_{s'=1}^{n'+1}\eta_{s'}\Big[\alpha_{s'}^{n'}E^{1}_{srs'\tilde{j}}+E^{2}_{srs'\tilde{j}}\Big]c_l^{\tilde{j}}P_l^{|\tilde{j}|}(\cos\varTheta_{s'}),\\
	&C_{srlj}=\sum_{|\tilde{j}|\leq l}D_{srl\tilde{j}}F_{sl\tilde{j}j}\mathrm{e}^{\mathrm{i}(j-\tilde{j})\varphi_r}, \quad B_{sj'lj}=\sum_{r=0}^{2n+3}C_{srlj}\mu_r\mathrm{e}^{-\mathrm{i}j'\varphi_r},\\
	&\mathbf{K}_{l'j',lj}=\sum_{s=1}^{n+2}B_{sj'lj}\nu_sc_{l'}^{j'}P_{l'}^{|j'|}(\cos\theta_{s}).
\end{align*}

Analogously,  the entry of $\mathbf{M}_{l'j',lj}^{k',\tilde{k}}$ can be obtained via the following operations:
\begin{align*}
	&E^{d,d'}_{srs'\tilde{j}}=\sum_{r'=0}^{2n'+1}\xi_{r'}\mathrm{e}^{\mathrm{i}\tilde{j}\varPhi_{r'}}\boldsymbol{v}^{(d')}(\theta_{s},\varphi_{r})^\top\mathcal{F}^\top(\hat{x}_{rs}) M_{n'}(\hat{x}_{rs},\hat{y}_{rs}^{r's'})\mathcal{F}(\hat{x}_{rs})T_{\hat{x}_{rs}}^{-1}\boldsymbol{v}^{(d)}(\varTheta_{s'},\varPhi_{r'}),\\
	&D^{\tilde{k},d'}_{srl\tilde{j}}=\sum_{s'=1}^{n'+1}\sum_{d=1}^{2}\eta_{s'}\alpha^{(\tilde{k},d)}_{l,\tilde{j}}(\varTheta_{s'})E^{d,d'}_{srs'\tilde{j}},\\
	&C^{\tilde{k},d'}_{srlj}=\sum_{|\tilde{j}|\leq l}F_{sl\tilde{j}j}\mathrm{e}^{\mathrm{i}(j-\tilde{j})\varphi_r}D^{\tilde{k},d'}_{srl\tilde{j}}, \quad B^{\tilde{k},d'}_{sj'lj}=\sum_{r=0}^{2n+3}\mu_r\mathrm{e}^{-\mathrm{i}j'\varphi_r}C^{\tilde{k},d'}_{srlj},\\
	&\mathbf{M}^{k',\tilde{k}}_{l'j',lj}=\sum_{s=1}^{n+2}\sum_{d'=1}^{2}\nu_s\overline{\alpha^{(k',d')}_{l',j'}(\theta_{s})}B^{\tilde{k},d'}_{sj'lj},
\end{align*}
where $M_{n'}(\hat{x}_{rs},\hat{y}_{rs}^{r's'})=\alpha_{s'}^{n'}\widetilde{M}_1(\hat{x}_{rs},\hat{y}_{rs}^{r's'})+\widetilde{M}_2(\hat{x}_{rs},\hat{y}_{rs}^{r's'})$, and $M_1(\hat{x},\hat{y})$, $M_2(\hat{x},\hat{y})$ are $3\times3$ matrices defined in \eqref{operatorsB}. In contrast to the operations in \cite{GH2008}, which put ${\boldsymbol{v}^{(d')}}^\top$ in $\boldsymbol{B}_{sj't'lj\tilde{k}}$, we combine ${\boldsymbol{v}^{(d')}}^\top$ and $\mathcal{F}^\top$ together so that $E^{d,d'}_{srs'\tilde{j}}$ is a scalar function, which makes the numerical implementation much easier since each operation is scalar.


\begin{thebibliography}{60}
	
\bibitem{ABG-15}
H. Ammari, E. Bretin, J. Garnier, H. Kang, H. Lee, and A. Wahab, Mathematical Methods in Elasticity Imaging, Princeton University Press, New Jersey, 2015.

\bibitem{BXY2017} 
G. Bao, L. Xu, and T. Yin, An accurate boundary element method for the exterior elastic scattering problem in two dimensions, J. Comput. Phys., 348 (2017), 343--363.

\bibitem{BronoYin2020} 
O. P. Bruno and T. Yin, Regularized integral equation methods for elastic scattering problems in three dimensions, J. Comput. Phys., 410 (2020), 109350.

\bibitem{BLR2014} 
F. Bu, J. Lin, and F. Reitich, A fast and high-order method for the three-dimensional elastic wave scattering problems, J. Comput. Phys., 258 (2014), 856--870.

\bibitem{DR-book1} 
D. Colton and R. Kress, Integral Equation Methods in Scattering Theory, SIAM, Philadelphia, 2013.

\bibitem{DR-book2} 
D. Colton and R. Kress, Inverse Acoustic and Electromagnetic Scattering Theory, Third Edition, Springer, New York, 2013.

\bibitem{DLL2021} 
H. Dong, J. Lai, and P. Li, A highly accurate boundary integral method for the elastic obstacle scattering problem, Math. Comp., 90 (2021), 2785--2814.

\bibitem{DLL2020} 
H. Dong, J. Lai, and P. Li, An inverse acoustic-elastic interaction problem with phased or phaseless far-field data, Inverse Problems, 36 (2020), 035014.

\bibitem{DLL2019} 
H. Dong, J. Lai, and P. Li, Inverse obstacle scattering for elastic waves with phased or phaseless far-field data, SIAM J. Imaging Sci., 12 (2019), 809--838.

\bibitem{GG2004} 
M. Ganesh and I. G. Graham, A high-order algorithm for obstacle scattering in three dimensions, J. Comput. Phys., 198 (2004), 211--242.

\bibitem{GH2007} 
M. Ganesh and S. C. Hawkins, A hybrid high-order algorithm for radar cross section computations, SIAM J. Sci. Comput., 29 (2007), 1217--1243.

\bibitem{GH2008} 
M. Ganesh and S. C. Hawkins, A high-order tangential basis algorithm for electromagnetic scattering by curved surface, J. Comput. Phys., 227 (2008), 4543--4562.

\bibitem{GS2002} 
I. G. Graham and I. H. Sloan, Fully discrete spectral boundary integral methods for Helmholtz problems on smooth closed surfaces in $\mathbb{R}^3$, Numer. Math., 92 (2002), 289--323.

\bibitem{GJ2019} 
L. Greengard and S. Jiang, A new mixed potential representation for the	equations of unsteady, incompressible flow, SIAM Review, 61 (2019), 733--755.

\bibitem{R-book2} 
R. Kress, Linear Integral Equations, Third Edition, Springer, 2010.

\bibitem{LL-86}
L. D. Landau and E. M. Lifshitz, Theory of Elasticity, Oxford: Pergamon 1986.

\bibitem{LY2019} 
P. Li and X. Yuan, Inverse obstacle scattering for elastic waves in three dimensions, Inverse Probl. Imaging., 13 (2019), 545--573.

\bibitem{LR1993} 
Y. Liu and F. J. Rizzo, Hypersingular boundary integral equations for radiation and scattering of elastic waves in three dimensions, Comput. Methods Appl. Mech. Engrg., 107 (1993), 131--144.

\bibitem{Louer2014} 
F. L. Lou\"{e}r, A high order spectral algorithm for elastic obstacle scattering in three dimensions, J. Comput. Phys., 279 (2014), 1--17.

\bibitem{Louer2020} 
F. L. Lou\"{e}r, A spectrally accurate method for the dielectric obstacle scattering problem and applications to the inverse problem, arXiv: 2006.10830, 2020.

\bibitem{MP-book1986} 
S. G. Mikhlin and S. Pr\"{o}ssdorf, Singular Integral Operators, Springer Verlag, Berlin, 1986.

\bibitem{Nedelec2001}
J. C. N\'{e}d\'{e}lec, Acoustic and Electromagnetic Equations: Integral Representations for Harmonic Problems, Springer, New York, 2000.

\bibitem{PV-JASA}
Y. H. Pao and V. Varatharajulu, Huygens' principle, radiation conditions, and integral formulas for the scattering of elastic waves, J. Acoust. Soc. Amer., 59 (1976), 1361--1371.

\bibitem{TC2007} 
M. S. Tong and W. C. Chew, Nystr\"{o}m method for elastic wave scattering by three-dimensional obstacles, J. Comput. Phys., 226 (2007), 1845--1858.

\bibitem{YLLY19}
J. Yue, M. Li, P. Li, and X. Yuan, Numerical solution of an inverse obstacle scattering problem for elastic waves via the Helmholtz decomposition, Commun. Comput. Phys., 26 (2019), 809--837. 

%
%
%
	

\end{thebibliography}
\end{document}